\documentclass[11pt]{amsart}

\usepackage{amsmath,amssymb,amsthm,amsfonts,mathrsfs}
\usepackage[all]{xy}
\CompileMatrices
\usepackage{mathdots}
\usepackage{ifthen}
\usepackage{fullpage}
\usepackage[usenames]{color}
\usepackage{comment}
\usepackage[shortlabels]{enumitem}
\usepackage[colorlinks=true,
linkcolor=blue,
anchorcolor=blue,
citecolor=red
]{hyperref}
\allowdisplaybreaks 

\newtheorem{thm}{Theorem}[section]
\newtheorem{prop}[thm]{Proposition}
\newtheorem{cor}[thm]{Corollary}
\newtheorem{lem}[thm]{Lemma}

\theoremstyle{definition}

\theoremstyle{remark}

\newcommand{\Z}{\mathbb{Z}}
\newcommand{\Q}{\mathbb{Q}}
\newcommand{\R}{\mathbb{R}}
\newcommand{\C}{\mathbb{C}}

\newcommand{\A}{\mathbb{A}}

\newcommand{\GL}{\mathrm{GL}}

\renewcommand{\Re}{\mathrm{Re}}

\newcommand{\Ind}{\mathrm{Ind}}

  1

\newcommand{\G}{\ifmmode {\mathcal{G}}\else${\mathcal{G}}$\ \fi}

\author{Yubo Jin}
\address{Institute for Advanced Study in Mathematics\\ Zhejiang University\\ Hangzhou, 310058, China}
\email{yubo.jin@durham.ac.uk}

\author{Pan Yan}
\address{Department of Mathematics, The University of Arizona, Tucson, AZ 85721, USA}
\email{panyan@math.arizona.edu}

\makeatletter
\@namedef{subjclassname@2020}{\textup{}2020 Mathematics Subject Classification}
\makeatother
\date{\today}
\title{Cohomology classes, periods, and special values of Rankin-Selberg $L$-functions}


\begin{document}

\begin{abstract}
In this article, we give a cohomological interpretation of (a special case of) the integrals constructed by the second named author and Q. Zhang \cite{YanZhang2023} which represent the product of Rankin-Selberg $L$-functions of $\mathrm{GL}_n\times\mathrm{GL}_m$ and $\mathrm{GL}_n\times\mathrm{GL}_{n-m-1}$ for $m<n$. As an application, we prove an algebraicity result for the special values of certain $L$-functions. This work is a generalization of the algebraicity result of Raghuram for $\mathrm{GL}_n\times\mathrm{GL}_{n-1}$ \cite{Raghuram2010} in the special case $m=n-1$, and the results of Mahnkopf \cite{Mahnkopf1998, Mahnkopf2005} in the special case $m=n-2$.
\end{abstract}

\setcounter{tocdepth}{1}

\maketitle
\tableofcontents

\section{Introduction}

The study of special values of $L$-functions is one of the central problems in number theory. In his celebrated paper \cite{Del79}, Deligne conjectures that the critical values of motivic $L$-functions, up to certain periods, are rational numbers. The main theme of this paper is the automorphic counterpart of his conjecture. That is, we study the algebraicity of special $L$-values for automorphic $L$-functions. 

The starting point of our research is Shimura's work on the special $L$-values for modular forms \cite{Sh76, Sh77} (and Hilbert modular forms \cite{Sh78}). See also \cite{Raghuram2011} whose setup and treatment are closer to the present paper. Later, algebraicity results for the critical values of Rankin-Selberg $L$-functions for $\mathrm{GL}_n\times\mathrm{GL}_{n-1}$ are obtained in \cite{KastenSchmidt, KazhdanMazurSchmidt, Raghuram2010, Raghuram2016}. Their works are based on the Rankin-Selberg convolution integrals developed by Jacquet--Piatetski-Shapiro--Shalika \cite{JPSS1983, JacquetShalika1981A, JacquetShalika1981B} (see also Cogdell's notes \cite{Cogdell2004Fields, Cogdell2007ParkCity} for a survey). More precisely, they consider the global integral of the form
\begin{equation}
\label{rankinselbergintegral}
\int_{\mathrm{GL}_{n-1}(F)\backslash\mathrm{GL}_{n-1}(\mathbb{A}_F)}\phi\left(\begin{bmatrix} g& \\ &1\end{bmatrix}\right)\phi'(g)|\det g|^{s-1/2}dg,
\end{equation}
where $\phi$ (resp. $\phi'$) is a cusp form in a cuspidal representation $\pi$ (resp. $\pi^\prime$) of $\mathrm{GL}_n(\mathbb{A}_F)$ (resp. $\mathrm{GL}_{n-1}(\mathbb{A}_F)$), which represents the Rankin-Selberg $L$-function $L(s,\pi\times\pi')$, and their algebraicity results are obtained by providing a cohomological interpretation of the integral \eqref{rankinselbergintegral}. 

In \cite{Mahnkopf1998}, Mahnkopf considers a variant of \eqref{rankinselbergintegral} when $n=3$, by replacing the cusp form $\phi^\prime$ with an Eisenstein series associated to a section of $\Ind_{B_2(\A_\Q)}^{\GL_2(\A_\Q)}(\chi_1\otimes\chi_2)$. Using a cohomological interpretation of the integral, he proves the algebraicity of $L(1,\pi\otimes\chi_1)/{\tilde{\Omega}(\pi_{\mathrm{f}})}$ for $\GL_3\times\GL_1$, where $\tilde{\Omega}(\pi_{\mathrm{f}})\in \mathbb{C}^\times$ denotes a certain period. Further developments are presented in Mahnkopf's subsequent work \cite{Mahnkopf2005}, where the algebraicity of the special value $L(0, \pi\times\chi)$ of $\GL_n\times\GL_1$, modulo certain period and Gauss sums, is obtained. The method he uses involves a variant of \eqref{rankinselbergintegral} using an Eisenstein series of $\GL_{n-1}$ associated to the standard parabolic subgroup of type $(n-2,1)$ as well as a cohomological interpretation of the integral.

Recently, a generalization of the Jacquet--Piatetski-Shapiro--Shalika's Rankin--Selberg convolution integrals is obtained in \cite{YanZhang2023}. To give more details, let $\pi$ (resp. $\tau_1$, $\tau_2$) be irreducible cuspidal representations of $\GL_n(\A_F)$ (resp. $\GL_m(\A_F)$, $\GL_k(\A_F)$), with $n>m+k$. Let $j$ be an integer such that $0\le j \le n-m-k-1$. The integrals considered in \cite{YanZhang2023} take the form
\begin{equation}
\label{eq-product-RS}
\int_{\GL_{m+k}(F)\backslash \GL_{m+k}(\A_F)}\phi^{\psi}_{Y_j}(\iota_j(h))E(h,f_{s_1, s_2, \tau_1, \tau_2})dh.
\end{equation}
Here, $\iota_j: \GL_{m+k}\to \GL_n$ is a certain embedding, $\phi\in \pi$ is a cusp form, $\phi^{\psi}_{Y_j}$ is a certain Fourier coefficient of $\phi$ along certain subgroup $Y_j\subset \GL_n$, and $E(h,f_{s_1, s_2, \tau_1, \tau_2})$ is the standard Eisenstein series on $\GL_{m+k}(\A_F)$ associated with a section $f_{s_1, s_2, \tau_1, \tau_2}$ in the representation induced from $\tau_1|\det|^{s_1-1/2}\otimes \tau_2|\det|^{-s_2+1/2}$ on the Levi subgroup of $\GL_{m+k}$ with partition $(m,k)$. This integral is Eulerian, and it represents
\begin{equation*}
\frac{L^S(s_1+\frac{n-m-k-1-2j}{2},\pi\times \tau_1)L^S(s_2-\frac{n-m-k-1-2j}{2},\widetilde\pi\times \widetilde\tau_2)}{L^S(s_1+s_2,\tau_1\times \widetilde \tau_2)}.
\end{equation*}
In the case $m=n-1$ and $k=0$, the integral \eqref{eq-product-RS} degenerates to the integral \eqref{rankinselbergintegral}.

The purpose of this paper is to provide a cohomological interpretation of the integral \eqref{eq-product-RS} in the special case $k=n-m-1$, and prove an algebraic result for the critical values of certain $L$-functions, generalizing the results of \cite{Mahnkopf1998, Mahnkopf2005, Raghuram2010}. We work exclusively over $F=\Q$ and let $\A=\A_\Q$. Our main result is the following.

\begin{thm}
\label{mainthm}
Let $\pi$ (resp. $\tau_1$, resp. $\tau_2$) be a regular algebraic cuspidal automorphic representation of $\mathrm{GL}_n(\mathbb{A})$ (resp. $\mathrm{GL}_m(\mathbb{A})$, resp. $\mathrm{GL}_k(\mathbb{A})$) with $n=m+k+1$ and $mk$ even. We assume $\tau_1,\tau_2$ have trivial central characters and certain conditions on the weights of $\pi,\tau_1,\tau_2$ (see the beginning of Section \ref{section-mainthm}). Then for any $\sigma\in\mathrm{Aut}(\C)$, we have
\begin{equation}
\label{maingaloisbehavior}
\begin{aligned}
&\sigma\left(\frac{L^S(\frac{1-k}{2},\pi\times\tau_1)L^S(\frac{1-m}{2},\widetilde{\pi}\times\widetilde{\tau_2})}{L^S(1+\frac{1-n}{2},\tau_1\times\widetilde{\tau_2})p^{\epsilon}(\pi)p^{\epsilon'}(\tau_1)p^{\epsilon'}(\tau_2)p_{\infty}(\pi,\tau_1,\tau_2)}\right)\\
=&\frac{L^S(\frac{1-k}{2},\pi^{\sigma}\times\tau^{\sigma}_1)L^S(\frac{1-m}{2},\widetilde{\pi}^{\sigma}\times\widetilde{\tau_2}^{\sigma})}{L^S(1+\frac{1-n}{2},\tau^{\sigma}_1\times\widetilde{\tau_2}^{\sigma})p^{\epsilon}(\pi^{\sigma})p^{\epsilon'}(\tau^{\sigma}_1)p^{\epsilon'}(\tau^{\sigma}_2)p_{\infty}(\pi,\tau_1,\tau_2)}.
\end{aligned}    
\end{equation}
In particular,
\begin{equation*}
\frac{L^S(\frac{1-k}{2},\pi\times\tau_1)L^S(\frac{1-m}{2},\widetilde{\pi}\times\widetilde{\tau_2})}{L^S(1+\frac{1-n}{2},\tau_1\times\widetilde{\tau_2})p^{\epsilon}(\pi)p^{\epsilon'}(\tau_1)p^{\epsilon'}(\tau_2)p_{\infty}(\pi,\tau_1,\tau_2)}\in\Q(\pi,\tau_1,\tau_2).
\end{equation*}
Here $p^{\epsilon}(\pi),p^{\epsilon'}(\tau_1),p^{\epsilon'}(\tau_2)$ are periods associated to cuspidal representations defined in Section \ref{cuspidalcohomology} and $p_{\infty}(\pi,\tau_1,\tau_2)$ is the archimedean period defined in \eqref{archimedeanperiod} where $\epsilon,\epsilon'$ are permissible signs satisfying $\epsilon=(-1)^n\epsilon'$ (uniquely determined by the parity of $m,k$).
\end{thm}

The assumption on the central characters of $\tau_1,\tau_2$ is only for simplicity in order to ease the notation and presentation of the paper. One can drop this assumption and the central characters will appear as Gauss sums in the denominator of \eqref{maingaloisbehavior} (as in for example \cite[Theorem 1.1]{Raghuram2010}). The assumptions on the weights of $\pi,\tau_1,\tau_2$ are used to guarantee the special values considered in \eqref{maingaloisbehavior} are critical and the cohomological pairing in Section \ref{pairing} makes sense. We also need the nonvanishing hypothesis (Section \ref{archimedeanpair}) as in for example \cite[Hypothesis 3.10]{Raghuram2010} and \cite{Mahnkopf1998, Mahnkopf2005}. This hypothesis is proved in \cite{Sun} in the situation of \emph{loc.cit} and in \cite{LLS22} for the cases treated in this paper. Indeed, more general algebraicity result is obtained in \cite[Theorem 1.2]{LLS22} for $\pi\times\tau$ with $\pi$ a cuspidal representation of $\mathrm{GL}_n(\mathbb{A})$ and $\tau$ a tamely isobaric representation of $\mathrm{GL}_{n-1}(\mathbb{A})$. Our approach is different to \cite{LLS22} by applying the explicit integral representations of \cite{YanZhang2023}.

In the above main theorem, we have only considered a specific critical value of the partial $L$-functions outside a finite set $S$. We can also obtain the result for the whole finite $L$-function (see Corollary \ref{finiteL}) by applying \cite[Proposition 3.17]{Raghuram2010} and the technique there. Furthermore, using the period relations in \cite{RaghuramShahidi2008}, we may also obtain results for twisted $L$-functions and other critical values. As an example, in Corollary \ref{othercriticalvalues}, we shift our critical values to $\frac{1}{2},0,1$. Together with the induction process as in \cite{Mahnkopf2005}, one can deduce an algebraicity result for some $\mathrm{GL}_n\times\mathrm{GL}_m$. However, the period \eqref{inductionperiod} obtained in the induction process is not so clear as it is related to other representations appeared through the induction (see the discussion at the end of Section \ref{section-mainthm}).

Once the algebraicity result for $L$-values is obtained, one can also ask for the $p$-adic interpolation of these $L$-values. For example, the $p$-adic $L$-function for $\mathrm{GL}_n\times\mathrm{GL}_{n-1}$ is constructed in \cite{Jan1, Jan2, Jan3, Jan4} and the $p$-adic $L$-function for $\mathrm{GL}_3$ is constructed in \cite{loeffler2021padic}. The $p$-adic interpolation for $L$-values appeared in \cite{Mahnkopf2005} and Theorem \ref{mainthm} will be studied in the authors' future research.

The rest of the paper, towards the proof of Theorem \ref{mainthm}, is organized as follows. We rephrase Yan and Zhang's integral in Section \ref{section-integral} with explicit choice of certain Whittaker functions and sections of the Eisenstein series. In Section \ref{cohomologyclass}, we review the cohomology theory required for the cohomological interpretation and attach cohomology classes to our chosen Whittaker function and Eisenstein series. The cohomological interpretation is then given in Section \ref{cohomologicalinterpretation} with the main theorem proved in Section \ref{section-mainthm}.

\subsection*{Acknowledgements}
This project was initiated due to a question raised to the second-named author by A. Raghuram during the Texas-Oklahoma Representations and Automorphic Forms Conference at University of Oklahoma in October 2023, and we would like to thank him for asking the question, and for helpful communications. We thank Binyong Sun for pointing us to the paper \cite{LLS22} and helpful discussions.

PY was partially supported by an AMS-Simons Travel Grant.

\section{Integral Representations of $L$-functions}
\label{section-integral}

In this section, we recall (a special case of) integral representations for product $L$-functions constructed in \cite{YanZhang2023}. We begin by fixing some general notations throughout the paper. The base number field considered in this paper is restricted to $\Q$ for simplicity. Denote $\mathbb{A}$ for its adele ring and $\mathbb{A}_{\mathrm{f}}$ the ring of finite adeles. Let $B_n=T_nU_n$ be the standard Borel subgroup of $\mathrm{GL}_n$ of all upper triangular matrices, $U_n$ the unipotent radical of $B_n$ and $T_n$ the diagonal torus. The center of $\mathrm{GL}_n$ will be denoted by $Z_n$. For an automorphic representation $\pi=\otimes_v\pi_v$ of $\mathrm{GL}_n(\mathbb{A})$, we write $\pi=\pi_{\infty}\otimes\pi_{\mathrm{f}}$ with $\pi_{\infty}$ a representation of $\mathrm{GL}_n(\R)$ and $\pi_{\mathrm{f}}=\otimes_{v\neq\infty}\pi_v$ a representation of $\mathrm{GL}_n(\mathbb{A}_{\mathrm{f}})$. For a $L$-function, we use the superscript $L^S$ to indicate the partial $L$-function outside a finite set $S$ of places and use the subscript $L_{\mathrm{f}}$ to indicate the finite part. 

\subsection{The global integral}

Let $P_{m,k}$ with $k=n-m-1$ be the standard parabolic subgroup of $\mathrm{GL}_{n-1}$ with Levi decomposition $P_{m,k}=M_{m,k}N_{m,k}$ and the Levi subgroup $M_{m,k}$ is isomorphic to $\mathrm{GL}_m\times\mathrm{GL}_k$. Let $(\tau_1,V_1)$ (resp. $(\tau_2,V_2)$) be an irreducible cuspidal automorphic representation of $\mathrm{GL}_{m}(\mathbb{A})$ (resp. $\mathrm{GL}_k(\mathbb{A})$). For a pair of complex numbers $(s_1,s_2)$, we consider the induced representation
\begin{equation}
\label{inducedrepresentation}
\mathrm{Ind}^{\mathrm{GL}_{n-1}(\mathbb{A})}_{P_{m,k}(\mathbb{A})}\left(\tau_1|\cdot|^{s_1-\frac{1}{2}}\otimes\tau_2|\cdot|^{-s_2+\frac{1}{2}}\right).
\end{equation} 
For a section $f_{s_1,s_2}$ of above induced representation space (viewed as a $\C$-valued function), we define an Eisenstein series on $\mathrm{GL}_{n-1}(\mathbb{A})$ by
\begin{equation*}
E(h;f_{s_1,s_2})=\sum_{\gamma\in P_{m,k}(\Q)\backslash\mathrm{GL}_{n-1}(\Q)}f_{s_1,s_2}(\gamma h),\qquad h\in\mathrm{GL}_{n-1}(\mathbb{A}).
\end{equation*}

Let $(\pi,V_{\pi})$ be an irreducible cuspidal automorphic representation of $\mathrm{GL}_n(\mathbb{A})$. Let $\iota:\mathrm{GL}_{n-1}\to\mathrm{GL}_n$ be an embedding given by
\begin{equation*}
\iota:\left[\begin{array}{cc}
a & b\\
c & d
\end{array}\right]\mapsto\left[\begin{array}{ccc}
a & & b\\
 & 1 & \\
c & & d
\end{array}\right],
\end{equation*}
where $a,b,c,d$ are of sizes $m\times m, m\times k, k\times m, k\times k$. Given a cusp form $\phi\in\pi$ and a section $f_{s_1, s_2}$, the global integral constructed in \cite{YanZhang2023} is defined as
\begin{equation}
\label{eq-globalintegral}
I(\phi,f_{s_1,s_2})=\int_{\mathrm{GL}_{n-1}(\Q)\backslash\mathrm{GL}_{n-1}(\mathbb{A})}\phi(\iota(h))E(h,f_{s_1,s_2})dh.
\end{equation}

This integral converges absolutely and uniformly in vertical strips in $\C$ for each variable $s_1$ and $s_2$, away from the poles of the Eisenstein series. 

We fix a non-trivial additive character $\psi:\Q\backslash\mathbb{A}\to\C^{\times}$ throughout the paper. Let $\mathcal{W}(\pi):=\mathcal{W}(\pi,\psi)$ be the $\psi$-Whittaker model of $\pi$, and write $\mathcal{W}(\pi)=\otimes_v\mathcal{W}(\pi_v)$, where $\mathcal{W}(\pi_v)$ is the $\psi_v$-Whittaker model of $\pi_v$. We also let $\mathcal{W}(\tau_1):=\mathcal{W}(\tau_1,\psi^{-1})$ (resp. $\mathcal{W}(\tau_2):=\mathcal{W}(\tau_2,\psi^{-1})$) be the $\psi^{-1}$-Whittaker model of $\tau_1$ (resp. $\tau_2$) and write $\mathcal{W}(\tau_1)=\otimes_v\mathcal{W}(\tau_{1,v})$ (resp. $\mathcal{W}(\tau_2)=\otimes_v\mathcal{W}(\tau_{2,v})$) with $\mathcal{W}(\tau_{1,v})$ (resp. $\mathcal{W}(\tau_{2,v})$) the local $\psi_v^{-1}$-Whittaker model of $\tau_{1,v}$ (resp. $\tau_{2,v}$). 

Let $W_\phi$ be the $\psi$-Whittaker function of $\phi$. In the region of absolute convergence, the integral $I(\phi,f_{s_1,s_2})$ unfolds to
\begin{equation}
\label{eq-unfolding}
I(\phi,f_{s_1,s_2})=\Psi(W_{\phi},f^{\mathcal{W}}_{s_1,s_2}):=\int_{U_{n-1}(\mathbb{A})\backslash\mathrm{GL}_{n-1}(\mathbb{A})}W_{\phi}(\iota(h))f^{\mathcal{W}}_{s_1,s_2}(h)dh,
\end{equation}
where
\begin{equation*}
f^{\mathcal{W}}_{s_1,s_2}(h)=\int_{U_m(\Q)\backslash U_m(\mathbb{A})\times U_k(\Q)\backslash U_k(\mathbb{A})}f_{s_1,s_2}\left(\left[\begin{array}{cc}
u_1 & 0\\
0 & u_2
\end{array}\right]h\right)\psi(u_1)\psi(u_2)du_1du_2.
\end{equation*}

We take $\phi=\otimes_v\phi_v$ to be a pure tensor and write $W_{\phi}=\otimes_v W_{\phi,v}$ where $W_{\phi,v}$ is a local Whittaker function of $\pi_v$. We also take $f_{s_1,s_2}=\otimes_vf_{s_1,s_2,v}$ (and thus $f^{\mathcal{W}}_{s_1,s_2}=\otimes_vf^{\mathcal{W}}_{s_1,s_2,v}$) to be decomposable with $f_{s_1,s_2,v}\in \Ind_{P_{m,k}(\Q_v)}^{\GL_{n-1}(\Q_v)} (\tau_{1,v}|\cdot|_v^{s_1-\frac{1}{2}}\otimes\tau_{2,v}|\cdot|_v^{-s_2+\frac{1}{2}})$ and $f_{s_1,s_2,v}^\mathcal{W}\in \Ind_{P_{m,k}(\Q_v)}^{\GL_{n-1}(\Q_v)} (\mathcal{W}(\tau_{1,v})|\cdot|_v^{s_1-\frac{1}{2}}\otimes \mathcal{W}(\tau_{2,v})|\cdot|_v^{-s_2+\frac{1}{2}})$. Then for $\Re(s_i)\gg 0$ for $i=1, 2$, the integral $\Psi(W_{\phi},f^{\mathcal{W}}_{s_1,s_2})$ has an Euler product decomposition (see \cite[Theorem 2.4]{YanZhang2023})
\begin{equation*}
	\Psi(W_{\phi},f^{\mathcal{W}}_{s_1,s_2})=\prod_v\Psi_v(W_{\phi,v},f_{s_1,s_2,v}^{\mathcal{W}})
\end{equation*}
where
\begin{equation*}
\Psi_v(W_{\phi,v},f_{s_1,s_2,v}^{\mathcal{W}}):=\int_{U_{n-1}(\Q_v)\backslash\mathrm{GL}_{n-1}(\Q_v)}W_{\phi,v}(\iota(h))f^{\mathcal{W}}_{s_1,s_2,v}(h)dh.	
\end{equation*}

At a local place $v$, the local zeta integral $\Psi_v(W_{\phi,v},f_{s_1,s_2,v}^{\mathcal{W}})$ is absolutely convergent for $\Re(s_i)\gg 0$ for $i=1, 2$. Over nonarchimedean local fields $\Q_v$, there exist $W_{\phi,v}$ and $f_{s_1,s_2,v}^{\mathcal{W}}$ such that $\Psi_v(W_{\phi,v},f_{s_1,s_2,v}^{\mathcal{W}})$ is absolutely convergent and equals 1, for all $(s_1, s_2)\in \C^2$. Over archimedean local field $\mathbb{R}$, for any $(s_1, s_2)\in \C^2$, there are choices of data $(W_{\phi,\infty}^j, f_{s_1,s_2,\infty}^{\mathcal{W},j})$ such that $\sum_j \Psi_\infty(W_{\phi,\infty}^j,f_{s_1,s_2,\infty}^{\mathcal{W},j})$ is holomorphic and nonzero in a neighborhood of $(s_1, s_2)$. See \cite[Proposition 3.4]{YanZhang2023}.

\subsection{Explicit choice of vectors}
\label{subsection-choiceofvectors}
In this subsection we make careful choice of local vectors $W_{\phi,v}$ and $f_{s_1,s_2,v}^{\mathcal{W}}$ in the local zeta integrals. 

Let $\mathfrak{f}_{\pi}$ (resp. $\mathfrak{f}_1$, resp. $\mathfrak{f}_2$) be the conductor of $\pi$ (resp. $\tau_1$, resp. $\tau_2$). We assume $\tau_1,\tau_2$ have trivial central characters for simplicity. Set $\mathfrak{f}_{\tau_1, \tau_2}=\mathrm{gcd}(\mathfrak{f}_1,\mathfrak{f}_2)=:\prod_v\mathfrak{f}_{\tau_1,\tau_2,v}$ and $\mathfrak{f}=\mathrm{gcd}(\mathfrak{f}_{\pi},\mathfrak{f}_{\tau_1,\tau_2})=:\prod_v\mathfrak{f}_v$. Let $S=\{v:v|\mathfrak{f}\text{ or }v=\infty\}$ be a finite set of places. Let $K_0(\mathfrak{f})\subset\mathrm{GL}_{n-1}(\mathbb{A}_f)$ be an open compact subgroup defined as
\begin{equation*}
\begin{aligned}
K_0(\mathfrak{f})&=\prod_{v<\infty}K_{0,v}(\mathfrak{f}_v),\\
K_{0,v}(\mathfrak{f}_v)&=\{\gamma\in\mathrm{GL}_{n-1}(\Z_v):\gamma\text{ mod }\mathfrak{f}_v\in B_{n-1}(\Z_v)\}.
\end{aligned}
\end{equation*}
That is, $K_{0,v}(\mathfrak{f}_v)$ contains matrices in $\gamma\in\mathrm{GL}_{n-1}(\Z_v)$ whose entries below the diagonal lie in $\mathfrak{f}_v\Z_v$.

For a finite place $v$, we define $W^0_{\pi,v}$ to be the normalized new vector as in \cite[Section 3.1.3]{Raghuram2010}. That is, we take $W^0_{\pi,v}$ be the new vector normalized such that $W^0_{\pi,v}(t_{\pi,v})=1$ where $t_{\pi,v}$ as in \cite[Lemma 1.3.2]{Mahnkopf2005}. If $\pi_v$ is unramified, then we take $t_{\pi,v}=1$ so that $W_{\pi,v}^0=W_{\pi,v}^{\mathrm{sp}}$ is the spherical vector. For any $\sigma\in\mathrm{Aut}(\C)$, we take $t_{\pi^{\sigma},v}=t_{\pi,v}$ so that $^{\sigma}W_{\pi,v}^0=W^0_{\pi^{\sigma},v}$. We similarly define the normalized new vectors $W^0_{\tau_1,v}, W^0_{\tau_2,v}$ for $\tau_1,\tau_2$ and denote $t_{\tau}=\mathrm{diag}[t_{\tau_1},t_{\tau_2}]$.

$\text{ }$

For each place $v$, we make the choice of $W_{\phi,v}$ as follows.
\begin{itemize}
\item{If $v\nmid\mathfrak{f}\infty$, we take $W_{\phi,v}=W_{\pi,v}^0$ to be the normalized new vector.}
\item{If $v|\mathfrak{f}$, we let $W_{\phi,v}$ be a Whittaker function whose restriction to $\mathrm{GL}_{n-1}(\Q_v)$ is supported on $U_{n-1}(\Q_v)t_{\tau,v}K_{0,v}(\mathfrak{f}_v)$ and 
\begin{equation*}
W_{\phi,v}\left(ut_{\tau,v}\gamma\right)=\psi(u),
\end{equation*}
for all $u\in U_{n-1}(\Q_v)$ and $\gamma\in K_{0,v}(\mathfrak{f}_v)$.
}
\item{If $v=\infty$, we take arbitrary non-zero $W_{\phi,\infty}$.}
\end{itemize}
Here we note that, by \cite[(2.3),(2.4), Proposition 3.2]{JacquetShalika1983} (see also \cite[Theorems E, F]{GelfandKajdan}), the Whittaker function on $\mathrm{GL}_n(\Q_v)$ is uniquely determined by its restriction to $\mathrm{GL}_{n-1}(\Q_v)$ and our choice made in the case $v|\mathfrak{f}$ exists.

$\text{ }$

For each place $v$, we make the choice of $f^{\mathcal{W}}_{s_1,s_2,v}$ as follow.
\begin{itemize}
\item{If $v\nmid\mathfrak{f}_{\tau_1,\tau_2}\infty$, we take $f^{\mathcal{W}}_{s_1,s_2,v}$ to be the spherical function. That is
\begin{equation*}
f^{\mathcal{W}}_{s_1,s_2,v}(g)=|\det(m_1)|^{s_1+\frac{k-1}{2}}|\det(m_2)|^{-s_2+\frac{1-m}{2}}W_{\tau_1,v}^0(m_1)W_{\tau_2,v}^0(m_2),
\end{equation*}
if we write $g=mn\gamma$ with $m=\mathrm{diag}[m_1,m_2]\in M_{m,k}(\Q_v)$, $n\in N_{m,k}(\Q_v),\gamma\in\mathrm{GL}_{n-1}(\Z_v)$.
}
\item{If $v|\mathfrak{f}_{\tau_1, \tau_2}$, we take $f^{\mathcal{W}}_{s_1,s_2,v}$ to be a function supported on $P_{m,k}(\Q_v)K_{0,v}(\mathfrak{f}_{\tau_1,\tau_2,v})$, and 
\begin{equation*}
\begin{aligned}
f^{\mathcal{W}}_{s_1,s_2,v}(g)=|\det(m_1)|^{s_1+\frac{k-1}{2}}|\det(m_2)|^{-s_2+\frac{1-m}{2}} W_{\tau_1,v}^0(m_1)W_{\tau_2,v}^0(m_2),
\end{aligned}
\end{equation*}
if we write $g=mn\gamma$ with $m=\mathrm{diag}[m_1,m_2]\in M_{m,k}(\Q_v)$, $n\in N_{m,k}(\Q_v)$, $\gamma\in K_{0,v}(\mathfrak{f}_{\tau_1, \tau_2, v})$.
}
\item{If $v=\infty$, we take arbitrary non-zero $f^{\mathcal{W}}_{s_1,s_2,\infty}$.}
\end{itemize}

$\text{ }$

Here we do not make the choice of vectors at archimedean place. Later they will be taken as cohomological vectors. We will specialize to the case $s_1=\frac{1-k}{2}$ and $s_2=\frac{1-m}{2}$ and we denote
\begin{equation*}
f_{\tau_1,\tau_2}^{\mathcal{W}}:=f_{\frac{1-k}{2},\frac{1-m}{2}}^{\mathcal{W}},\qquad f_{\tau_1,\tau_2,v}^{\mathcal{W}}:=f_{\frac{1-k}{2},\frac{1-m}{2},v}^{\mathcal{W}}.
\end{equation*}
We also set the finite part of our vectors as
\begin{equation*}
W_{\phi,\mathrm{f}}=\prod_{v\neq \infty}W_{\phi,v},\qquad f^{\mathcal{W}}_{\tau_1,\tau_2,\mathrm{f}}=\prod_{v\neq \infty}f^{\mathcal{W}}_{\tau_1,\tau_2,v}.
\end{equation*}

We need to calculate the local integrals for our chosen vectors. The unramified computations (for $v\nmid\mathfrak{f}$) is done in \cite[Proposition 3.8]{YanZhang2023}. We calculate the ramified local integrals in the following.

\begin{lem}
\label{ramifiedcomputation}
Let $v|\mathfrak{f}$, then with our choice of $W_{\phi,v},f^{\mathcal{W}}_{s_1,s_2,v}$ above, we have
\begin{equation*}
\Psi_v(W_{\phi,v},f^{\mathcal{W}}_{s_1,s_2,v})=\mathrm{vol}(K_{0,v}(\mathfrak{f}_v))|\det(t_{\tau_1,v})|^{s_1+\frac{k-1}{2}}|\det(t_{\tau_2,v})|^{-s_2+\frac{1-m}{2}}.
\end{equation*}
\end{lem}

\begin{proof}
By our definition of $W_{\phi,v}$, the integrand vanishes unless $h\in U_{n-1}(\Q_v)t_{\tau,v}K_{0,v}(\mathfrak{f}_v)$. Then clearly the integral equals
\[
\begin{aligned}
&\mathrm{vol}(K_{0,v}(\mathfrak{f}_v))|\det(t_{\tau_1,v})|^{s_1+\frac{k-1}{2}}|\det(t_{\tau_2,v})|^{-s_2+\frac{1-m}{2}}W_{\tau_1,v}^0(t_{\tau_1,v})W_{\tau_2,v}^0(t_{\tau_2,v})\\
=&\mathrm{vol}(K_{0,v}(\mathfrak{f}_v))|\det(t_{\tau_1,v})|^{s_1+\frac{k-1}{2}}|\det(t_{\tau_2,v})|^{-s_2+\frac{1-m}{2}}
\end{aligned}
\]
by the definition of the normalized new vectors.
\end{proof}

With the explicit choice of vectors $W_{\phi,v},f_{s_1,s_2,v}^{\mathcal{W}}$, we now rephrase \cite[Theorem 1.1]{YanZhang2023} as follows.

\begin{thm}\cite[Theorem 1.1]{YanZhang2023}
Let $S$ be a finite set of places containing $\infty$ outside of which all date are unramified.
Let $\phi$ and $f_{s_1,s_2}$ be decomposable, with the local vectors chosen above. Then we have
\begin{equation}
\label{integralrepresentation}
	I(\phi,f_{\frac{1-k}{2},\frac{1-m}{2}})=  \frac{L^S(\frac{1-k}{2},\pi\times \tau_1)L^S(\frac{1-m}{2},\widetilde\pi\times \widetilde{\tau}_2) }{L^S(1+\frac{1-n}{2}, \tau_1 \times \widetilde{\tau}_2)} \cdot \mathrm{vol}(K_0(\mathfrak{f}))
	 \cdot \Psi_\infty(W_{\phi,v},f_{\tau_1,\tau_2,v}^{\mathcal{W}}).
\end{equation}
\end{thm}

\section{The Cohomology Classes and Periods}
\label{cohomologyclass}

The aim of this section is to attach cohomology classes and periods to cuspidal representations and Eisenstein series. We closely follow \cite[Section 3,4]{Mahnkopf2005}. Though only the Eisenstein series induced from the parabolic subgroup of type $(n-2,1)$ are considered there, most arguments can be directly generalized with the supplement of \cite{HarderRaghuram2020book}.

\subsection{Notations and preliminaries}

We fix more notations and briefly review some  cohomology groups used in this section for the convenience of the reader. The reader may refer to \cite{HarderRaghuram2020book, Mahnkopf2005, RaghuramShahidi2008} for more details. 

Write $\mathrm{O}(n),\mathrm{SO}(n)$ for the orthogonal group and the special orthogonal group in $\mathrm{GL}_n(\R)$. Let $K_{n,\infty}=\mathrm{O}(n)Z_n(\R)$ be the maximal compact subgroup of $\mathrm{GL}_n(\R)$ thickened by the center $Z_n(\R)$. Denote $K_{n,\infty}^0$ for the connected component of $K_{n,\infty}$ and $K_{n,\infty}^1=\mathrm{SO}(n)$. For an open compact subgroup $K_{\mathrm{f}}$ of $\mathrm{GL}_n(\mathbb{A}_{\mathrm{f}})$, we consider the locally symmetric spaces
\begin{equation*}
S_n(K_{\mathrm{f}}):=\mathrm{GL}_n(\Q)\backslash\mathrm{GL}_n(\mathbb{A})/K_{n,\infty}^0K_{\mathrm{f}},\qquad F_n(K_{\mathrm{f}}):=\mathrm{GL}_n(\Q)\backslash\mathrm{GL}_n(\mathbb{A})/K_{n,\infty}^1K_{\mathrm{f}}.
\end{equation*}
Taking the limit over all open compact subgroup $K_{\mathrm{f}}$, we may also consider
\begin{equation*}
S_n:=\lim_{\substack{\longleftarrow\\K_{\mathrm{f}}}}S_n(K_{\mathrm{f}}),\qquad F_n:=\lim_{\substack{\longleftarrow\\K_{\mathrm{f}}}}F_n(K_{\mathrm{f}}).
\end{equation*}

Let $X^+(T_n)$ be the set of dominant integral weights of $T_n$ and $X_0^+(T_n)$ the subset of pure weights. For $\mu\in X^+(T_n)$, we denote by $M_{\mu}$ the irreducible representation of $\mathrm{GL}_n(\C)$ of highest weight $\mu$ and let $\mathcal{M}_{\mu}$ for the sheaf of $\Q$-vector spaces on $S_n(K_{\mathrm{f}})$ associated to $M_{\mu}$. We may also use the same notation when viewing $\mathcal{M}_{\mu}$ as a sheaf on $F_n(K_{\mathrm{f}}), S_n, F_n$. We will write $\mu^{\vee}$ for the dual weight of $\mu$ and $M_{\mu}^{\vee}=M_{\mu^{\vee}}$ the contragredient of $M_{\mu}$. 

We consider the sheaf cohomology group $H^{\bullet}(S_n(K_{\mathrm{f}}),\mathcal{M}_{\mu})$.
Let $\overline{S}_n(K_{\mathrm{f}})=S_n(K_{\mathrm{f}})\cup\partial S_n(K_{\mathrm{f}})$ be the Borel-Serre compactification of $S_n(K_{\mathrm{f}})$ where the boundary is stratified as $\partial S_n(K_{\mathrm{f}})=\cup_P\partial_P S_n(K_{\mathrm{f}})$ with $P$ running through the $\mathrm{GL}_n(\Q)$-conjugacy classes of proper parabolic subgroups defined over $\Q$. We will also consider the boundary cohomology group $H^{\bullet}(\partial S_n(K_{\mathrm{f}}),\mathcal{M}_{\mu})$ and we use the subscript ``$c$", i.e. $H_c^{\bullet}(S_n(K_{\mathrm{f}}),\mathcal{M}_{\mu})$, for the cohomology group with compact support. We can also similarly consider the sheaf cohomology groups on $F_n(K_{\mathrm{f}}), S_n, F_n$.

The sheaf cohomology group $H^{\bullet}(S_n(K_{\mathrm{f}}),\mathcal{M}_{\mu})$ can be identified with the Lie algebra cohomology group
\begin{equation*}
H^{\bullet}(S_n(K_{\mathrm{f}}),\mathcal{M}_{\mu})\cong H^{\bullet}(\mathfrak{g}_{n},K_{n,\infty}^0;\mathcal{A}(\mathrm{GL}_n(\Q)\backslash\mathrm{GL}_n(\mathbb{A}))^{K_{\mathrm{f}}}\otimes M_{\mu}).
\end{equation*}
Here $\mathfrak{g}_{n}$ is the Lie algebra of $\mathrm{GL}_n(\R)$ and $\mathcal{A}(\mathrm{GL}_n(\Q)\backslash\mathrm{GL}_n(\mathbb{A}))$ is the space of automorphic forms. The cuspidal cohomology group is defined as a subgroup of $H^{\bullet}(S_n(K_{\mathrm{f}}),\mathcal{M}_{\mu})$ such that
\begin{equation*}
H_{\mathrm{cusp}}^{\bullet}(S_n(K_{\mathrm{f}}),\mathcal{M}_{\mu})\cong H^{\bullet}(\mathfrak{g}_{n},K_{n,\infty}^0;\mathcal{A}_{\mathrm{cusp}}(\mathrm{GL}_n(\Q)\backslash\mathrm{GL}_n(\mathbb{A}))^{K_{\mathrm{f}}}\otimes M_{\mu}),
\end{equation*}
where $\mathcal{A}_{\mathrm{cusp}}(\mathrm{GL}_n(\Q)\backslash\mathrm{GL}_n(\mathbb{A}))$ is the space of cusp forms. The decomposition of $\mathcal{A}_{\mathrm{cusp}}(\mathrm{GL}_n(\Q)\backslash\mathrm{GL}_n(\mathbb{A}))$ into cuspidal automorphic representations provide a direct sum decomposition
\begin{equation*}
\begin{aligned}
&H^{\bullet}(\mathfrak{g}_{n},K_{n,\infty}^0;\mathcal{A}_{\mathrm{cusp}}(\mathrm{GL}_n(\Q)\backslash\mathrm{GL}_n(\mathbb{A}))^{K_{\mathrm{f}}}\otimes M^{\vee}_{\mu})\\
\cong&\bigoplus_{\pi\in\mathrm{Coh}(\mathrm{GL}_n,K_{\mathrm{f}},\mu)}H^{\bullet}(\mathfrak{g}_{n},K_{n,\infty}^0;\pi\otimes M_{\mu}^{\vee}).
\end{aligned}
\end{equation*}
We are writing $\mathrm{Coh}(\mathrm{GL}_n,K_{\mathrm{f}},\mu)$ for the set of those $\pi$ with nonzero contribution to above cohomology groups and we set $\mathrm{Coh}(\mathrm{GL}_n,\mu):=\cup_{K_{\mathrm{f}}}\mathrm{Coh}(\mathrm{GL}_n,K_{\mathrm{f}},\mu)$. The direct summand in above decomposition can be further decomposed into isotypic components along permissible signs $\epsilon\in\{\pm\}\cong(K_{n,\infty}/K_{n,\infty}^0)^{\wedge}$ (i.e. $\epsilon$ can be arbitrary when $n$ is even and $\epsilon$ is the central character of $\pi_{\infty}\otimes M_{\mu}$ at $-1$ if $n$ is odd):
\begin{equation*}
H^{\bullet}(\mathfrak{g}_{n},K_{n,\infty}^0;\pi\otimes M_{\mu}^{\vee})=\bigoplus_{\epsilon}H^{\bullet}(\mathfrak{g}_{n},K_{n,\infty}^0;\pi\otimes M_{\mu}^{\vee})(\epsilon).
\end{equation*}

\subsection{Cuspidal cohomology}
\label{cuspidalcohomology}

We now recall the definition of periods and cohomology classes associated to a regular algebraic cuspidal automorphic representation following \cite{RaghuramShahidi2008} and \cite[Section 3.2.1, 3.2.2]{Raghuram2010}.

Let $(\pi,V_{\pi})$ be a regular algebraic cuspidal automorphic representation of $\mathrm{GL}_n(\mathbb{A})$. Then there is a dominant integral pure weight $\mu\in X_0^+(T_n)$ such that $\pi\in\mathrm{Coh}(\mathrm{GL}_n,\mu^{\vee})$. Let $\epsilon\in\{\pm\}\cong(K_{n,\infty}/K_{n,\infty}^0)^{\wedge}$ be a permissible sign for $\pi$. Denote $b_n=n^2/4$ if $n$ is even and $b_n=(n^2-1)/4$ if $n$ is odd. We fix a generator $[\pi_{\infty}]$ for the one-dimensional $\C$-vector space $H^{b_n}(\mathfrak{g}_{n},K_{\infty}^0;V_{\pi,\infty}\otimes M_{\mu}^{\vee})(\epsilon)$. There is a $\mathrm{GL}_n(\mathbb{A}_{\mathrm{f}})$ equivariant map
\begin{equation}
\label{cohomologymap}
\mathcal{F}_{\pi_{\mathrm{f}},\epsilon,[\pi_{\infty}]}:\mathcal{W}(\pi_{\mathrm{f}})\to H^{b_n}(\mathfrak{g}_{n},K_{n,\infty}^0;V_{\pi}\otimes M_{\mu}^{\vee})(\epsilon).
\end{equation}
Recall that $\mathcal{W}(\pi)=\mathcal{W}(\pi,\psi)$ is the $\psi$-Whittaker model of $\pi$, and $\mathcal{W}(\pi_v),\mathcal{W}(\pi_{\mathrm{f}})$ are the local counterparts.

It is known in \cite[Theorem 3.1]{RaghuramShahidi2008} that $\pi_{\mathrm{f}}$ admits a $\Q(\pi)$-structure over a number field $\Q(\pi)$ (called the rationality field of $\pi$). Both vector spaces in the above map have $\Q(\pi)$-structures which are unique up to homotheties. On one hand, define an element $t_{\sigma}\in\mathbb{A}_{\mathrm{f}}^{\times}$ via the cyclotomic character
\begin{equation*}
\begin{array}{ccccccccc}
\mathrm{Aut}(\C) & \to & \mathrm{Gal}(\overline{\Q}/\Q) & \to & \mathrm{Gal}(\Q(\mu_{\infty})/\Q) & \to & \widehat{\Z}^{\times} & \cong & \prod_p\Z_p^{\times}\\
\sigma & \mapsto & \sigma|_{\overline{\Q}} & \mapsto & \sigma|_{\Q(\mu_{\infty})} & \mapsto & & & t_{\sigma}
\end{array}
\end{equation*}
Set $t_{\sigma,n}=\mathrm{diag}[t_{\sigma}^{-(n-1)},t_{\sigma}^{-(n-2)},...,1]\in\mathrm{GL}_n(\mathbb{A}_{\mathrm{f}})$. On the Whittaker model $\mathcal{W}(\pi_{\mathrm{f}})$, we have the following action of $\mathrm{Aut}(\C)$:
\begin{equation*}
^{\sigma}W(g_{\mathrm{f}}):=\sigma(W(t_{\sigma,n}g_{\mathrm{f}})),\qquad \sigma\in\mathrm{Aut}(\C), W\in\mathcal{W}(\pi_{\mathrm{f}}).
\end{equation*}
This action gives a $\Q(\pi)$-structure of $\mathcal{W}(\pi_{\mathrm{f}})$. Also note that the action makes sense locally by replacing $t_{\sigma}$ with $t_{\sigma,v}$. On the other hand, $H^{b_n}(\mathfrak{g}_{\infty},K_{\infty}^0;V_{\pi}\otimes M_{\mu}^{\vee})(\epsilon)$ has a $\Q(\pi)$-structure induced from the $\Q(\pi)$-structure of the Betti cohomology $H^{b_n}_{\mathrm{Betti}}(S_n,\mathcal{M}_{\mu}^{\vee})$. 

The period is then defined by the difference of these two $\Q(\pi)$-structures in \eqref{cohomologymap}. More precisely, the period $p^{\epsilon}(\pi)\in\C^{\times}/\Q(\pi)^{\times}$ is defined by a complex number such that the normalized map
\begin{equation*}
\mathcal{F}^0_{\pi_{\mathrm{f}},\epsilon,[\pi_{\infty}]}=p^{\epsilon}(\pi)^{-1}\mathcal{F}_{\pi_{\mathrm{f}},\epsilon,[\pi_{\infty}]},
\end{equation*}
is $\mathrm{Aut}(\C)$-equivariant, that is for all $\sigma\in\mathrm{Aut}(\C)$ one has
\begin{equation*}
\sigma\circ\mathcal{F}^0_{\pi_{\mathrm{f}},\epsilon,[\pi_{\infty}]}=\mathcal{F}^0_{\pi_{\mathrm{f}}^{\sigma},\epsilon,[\pi_{\infty}^{\sigma}]}\circ\sigma.
\end{equation*}

For the Whittaker functions $W_{\phi,{\mathrm{f}}}$ chosen in Section \ref{subsection-choiceofvectors}, we define the cohomology classes attached to it by
\begin{equation*}
\theta_{\pi,\epsilon}=\mathcal{F}_{\pi_{\mathrm{f}},\epsilon,[\pi_{\infty}]}(W_{\phi,{\mathrm{f}}}),\qquad\theta^0_{\pi,\epsilon}=\mathcal{F}^0_{\pi_{\mathrm{f}},\epsilon,[\pi_{\infty}]}(W_{\phi,{\mathrm{f}}})=p^{\epsilon}(\pi)^{-1}\theta_{\pi,\epsilon}.
\end{equation*}

Let $K_{\mathrm{f}}$ be an open compact subgroup of $\mathrm{GL}_n(\mathbb{A}_{\mathrm{f}})$ which fixes $W_{\phi,{\mathrm{f}}}$. We then obtain a class $\theta_{\pi,\epsilon}$ in $H_c^{b_n}(S_n(K_{\mathrm{f}}),\mathcal{M}_{\mu}^{\vee})$ via the maps
\begin{equation*}
\begin{aligned}
H^{b_n}(\mathfrak{g}_{n},K_{n,\infty}^0;V_{\pi}^{K_{\mathrm{f}}}\otimes M_{\mu}^{\vee})(\epsilon)&\to H^{b_n}_{\mathrm{cusp}}(S_n(K_{\mathrm{f}}),\mathcal{M}_{\mu}^{\vee})(\pi_{\mathrm{f}}\otimes\epsilon)\\
&\to H_c^{b_n}(S_n(K_{\mathrm{f}}),\mathcal{M}_{\mu}^{\vee}).
\end{aligned}
\end{equation*}

The inclusion map $\iota:\mathrm{GL}_{n-1}\to\mathrm{GL}_n$ induces a proper map $\iota:F_{n-1}(R_{\mathrm{f}})\to S_n(K_{\mathrm{f}})$ with $R_{\mathrm{f}}:=\iota^{\ast}K_{\mathrm{f}}$ which further induces a map
\begin{equation*}
\iota^{\ast}:H_c^{\bullet}(S_n(K_{\mathrm{f}}),\mathcal{M}_{\mu}^{\vee})\to H_c^{\bullet}(F_{n-1}(R_{\mathrm{f}}),\iota^{\ast}\mathcal{M}_{\mu}^{\vee}).
\end{equation*}
This provide us cohomology classes
\begin{equation*}
\iota^{\ast}\theta_{\pi,\epsilon}\in H_c^{b_n}(F_{n-1}(R_{\mathrm{f}}),\iota^{\ast}\mathcal{M}_{\mu}^{\vee}),\qquad\iota^{\ast}\theta^0_{\pi,\epsilon}\in H_c^{b_n}(F_{n-1}(R_{\mathrm{f}}),\iota^{\ast}\mathcal{M}_{\mu}^{\vee}).
\end{equation*}

\begin{lem}
\label{pibehavior}
For any $\sigma\in\mathrm{Aut}(\C)$, we have
\begin{equation*}
    ^{\sigma}\theta^0_{\pi,\epsilon}=\theta^0_{\pi^{\sigma},\epsilon}.
\end{equation*}
\end{lem}

\begin{proof}
The proof is similar to \cite[Proposition 3.15]{Raghuram2010}. By definition of the cohomology class and the period, we have
\[
^{\sigma}\theta^0_{\pi,\epsilon}={^{\sigma}\mathcal{F}}^0_{\pi_{\mathrm{f}},\epsilon,[\pi_{\infty}]}(W_{\phi,\mathrm{f}})=\mathcal{F}^0_{\pi^{\sigma}_{\mathrm{f}},\epsilon,[\pi_{\infty}]}({^{\sigma}W}_{\phi,\mathrm{f}}).
\]
It then suffices to show
\[
{^{\sigma}W}_{\phi,v}=W_{{\phi^{\sigma}},v},
\]
for all finite places $v$. For $v\nmid\mathfrak{f}$, it is clear by our normalization of the new vector. For $v|\mathfrak{f}$, note that the Whittaker function is uniquely determined by its restriction to $\mathrm{GL}_{n-1}(\Q_v)$ (\cite[Proposition 3.2]{JacquetShalika1983}). By our definition of $W_{\phi,v}$, the support of $^{\sigma}W_{\phi,v}$ when restricted to $\mathrm{GL}_{n-1}(\Q_v)$ is still the double coset $U_{n-1}(\Q_v)t_{\tau,v}K_{0,v}(\mathfrak{f}_v)$, and on this double coset we have 
\[
\begin{aligned}
^{\sigma}W_{\phi,v}(\iota(ut_{\tau,v}\gamma))&=\sigma\left(W_{\phi,v}\left(\iota\left(\left[\begin{array}{cccc}
t_{\sigma,v}^{-(n-1)} & & & \\
 & t_{\sigma,v}^{-(n-2)} & & \\
 & & \ddots & \\
 & & & t_{\sigma,v}^{-1}
\end{array}\right]ut_{\tau,v}\gamma\right)\right)\right)\\
&=\sigma(\psi_v(t_{\sigma,v}^{-1}u))=W_{\phi^{\sigma},v}(\iota(ut_{\tau,v}\gamma)).
\end{aligned}
\]
as desired.
\end{proof}

\subsection{Cohomology associated to the induced representation}
\label{cohomologyinduced}

Let $(\tau_1,V_1)$ (resp. $(\tau_2,V_2)$) be a regular algebraic cuspidal representation of $\mathrm{GL}_m(\mathbb{A})$ (resp. $\mathrm{GL}_k(\mathbb{A})$). Then there is a dominant integral pure weight $\mu_1\in X_0^+(T_m)$ (resp. $\mu_2\in X_0^+(T_k)$) such that $\tau_1\in\mathrm{Coh}(\mathrm{GL}_m,\mu_1^{\vee})$ (resp. $\tau_2\in\mathrm{Coh}(\mathrm{GL}_k,\mu_2^{\vee})$). We assume $\mu_1,\mu_2$ are balanced in the sense that there exists a balanced Kostant representative $w$ (see \cite[Definition 5.9]{HarderRaghuram2020book}) such that $\lambda=w^{-1}(\mu_1+\mu_2)$ is a dominant weight for $\mathrm{GL}_{n-1}$. Especially, we need to assume that $mk$ is even. In this subsection, we will consider the algebraic (un-normlized) induction $^a\mathrm{Ind}_{P_{m,k}(\mathbb{A})}^{\mathrm{GL}_{n-1}(\mathbb{A})}(\tau_1\otimes\tau_2)$. In particular, this equals the (normalized) induction $\mathrm{Ind}_{P_{m,k}(\mathbb{A})}^{\mathrm{GL}_{n-1}(\mathbb{A})}(\tau_1|\cdot|^{s_1-\frac{1}{2}}\otimes\tau_2|\cdot|^{-s_2+\frac{1}{2}})$ when specializing to $s_1=\frac{1-k}{2}$ and $s_2=\frac{1-m}{2}$.

Without loss of generality, we will assume $k$ is even and $m$ is an arbitrary integer. Everything in this subsection still holds for $m$ even and $k$ arbitrary by exchanging the place of $\tau_1,\tau_2$. Let $\epsilon'$ be a permissible sign for $\tau_1$. We fix a generator $[\tau_{1,\infty}]$ (resp. $[\tau_{2,\infty}]$) for the one-dimensional $\C$-vector space
\begin{equation*}
H^{b_m}(\mathfrak{g}_{m},K_{m,\infty}^0;V_{1,\infty}\otimes\mathcal{M}_{\mu_1}^{\vee})(\epsilon')\qquad \text{ resp. }H^{b_k}(\mathfrak{g}_{k},K_{k,\infty}^0;V_{2,\infty}\otimes\mathcal{M}_{\mu_2}^{\vee})(\epsilon').
\end{equation*}
By the Delorme's Lemma \cite[Theorem.III.3.3]{BorelWallach} (see also \cite[Section 7.3.2.3]{HarderRaghuram2020book}), we have the isomorphism
\begin{equation*}
\begin{aligned}
&H^{b_{n-1}}(\mathfrak{g}_{n-1},K^0_{n-1,\infty};{^a\mathrm{Ind}}_{P_{m,k}(\R)}^{\mathrm{GL}_{n-1}(\R)}(V_{1,\infty}\otimes V_{2,\infty})\otimes\mathcal{M}_{\lambda}^{\vee})(\epsilon')\\
\stackrel{\sim}\longrightarrow&H^{b_{m}}(\mathfrak{g}_{m},K^0_{m,\infty};V_{1,\infty}\otimes\mathcal{M}_{\mu_1}^{\vee})(\epsilon')\otimes H^{b_{k}}(\mathfrak{g}_{k},K^0_{k,\infty};V_{2,\infty}\otimes\mathcal{M}_{\mu_2}^{\vee})(\epsilon').
\end{aligned}
\end{equation*}
We thus obtain a generator
\begin{equation*}
[\tau_{1,\infty}]\otimes[\tau_{2,\infty}]\in H^{b_{n-1}}(\mathfrak{g}_{n-1},K^0_{n-1,\infty};{^a\mathrm{Ind}}_{P_{m,k}(\R)}^{\mathrm{GL}_{n-1}(\R)}(V_{1,\infty}\otimes V_{2,\infty})\otimes\mathcal{M}_{\lambda}^{\vee})(\epsilon').
\end{equation*}

Recall $\mathcal{W}(\tau_1)=\mathcal{W}(\tau_1,\psi^{-1})$ (resp. $\mathcal{W}(\tau_2)=\mathcal{W}(\tau_2,\psi^{-1})$) is the Whittaker model of $\tau_1$ (resp. $\tau_2$) associated to a fixed nontrivial additive character $\psi^{-1}$. We have the identification
\begin{equation*}
{^a\mathrm{Ind}}_{P_{m,k}(\mathbb{A})}^{\mathrm{GL}_{n-1}(\mathbb{A})}(V_1\otimes V_2)\stackrel{\sim}\longrightarrow{^a\mathrm{Ind}}_{P_{m,k}(\mathbb{A})}^{\mathrm{GL}_{n-1}(\mathbb{A})}(\mathcal{W}(\tau_1)\otimes\mathcal{W}(\tau_2)).
\end{equation*}
As in \eqref{cohomologymap}, we now have
\begin{equation}
\label{boundarymap}
\begin{aligned}
\mathcal{F}_{\tau_1,\tau_2,\epsilon'}:&{^a\mathrm{Ind}}_{P_{m,k}(\mathbb{A}_{\mathrm{f}})}^{\mathrm{GL}_{n-1}(\mathbb{A}_{\mathrm{f}})}(\mathcal{W}(\tau_{1,{\mathrm{f}}})\otimes\mathcal{W}(\tau_{2,{\mathrm{f}}}))\\
\longrightarrow& H^{b_{n-1}}(\mathfrak{g}_{n-1},K^0_{n-1,\infty};{^a\mathrm{Ind}}_{P_{m,k}(\mathbb{A})}^{\mathrm{GL}_{n-1}(\mathbb{A})}(V_{1}\otimes V_{2})\otimes M_{\lambda}^{\vee})(\epsilon')\\
\stackrel{\sim}\longrightarrow&H^{b_{n-1}}(\partial_{P_{m,k}}S_{n-1},\mathcal{M}_{\lambda}^{\vee})({^a\mathrm{Ind}}_{P_{m,k}(\mathbb{A})}^{\mathrm{GL}_{n-1}(\mathbb{A})}(V_{1}\otimes V_{2})\otimes\epsilon').
\end{aligned}
\end{equation}

The action of $\mathrm{Aut}(\C)$ on Whittaker models induces an action on the induced representation space ${^a\mathrm{Ind}}_{P_{m,k}(\mathbb{A}_{\mathrm{f}})}^{\mathrm{GL}_{n-1}(\mathbb{A}_{\mathrm{f}})}(\mathcal{W}(\tau_{1,{\mathrm{f}}})\otimes\mathcal{W}(\tau_{2,{\mathrm{f}}}))$. The vector spaces in the above map have $\Q(\tau_1,\tau_2)$-structures. We are going to define a period to normalize the above map such that it is $\mathrm{Aut}(\C)$-equivariant. Let $p^{\epsilon'}(\tau_1)$ (resp. $p^{\epsilon'}(\tau_2)$) be the period associated to $\tau_1$ (resp. $\tau_2$) defined as in Section \ref{cuspidalcohomology} and set
\begin{equation}
\label{normalizeboundary}
\mathcal{F}^0_{\tau_1,\tau_2,\epsilon'}:=p^{\epsilon'}(\tau_1)^{-1}p^{\epsilon'}(\tau_2)^{-1}\mathcal{F}_{\tau_1,\tau_2,\epsilon'}.
\end{equation}

\begin{lem}
The normalized map $\mathcal{F}^0_{\tau_1,\tau_2,\epsilon'}$ is $\mathrm{Aut}(\C)$-equivariant. That is, for all $\sigma\in\mathrm{Aut}(\C)$, one has
\begin{equation*}
\sigma\circ\mathcal{F}^0_{\tau_1,\tau_2,\epsilon'}=\mathcal{F}^0_{\tau_1,\tau_2,\epsilon'}\circ\sigma.
\end{equation*}
\end{lem}

\begin{proof}
The proof is similar to the argument in \cite[Section 4.4.1]{Mahnkopf2005}. The map \eqref{cohomologymap} induces
\[
\begin{aligned}
\mathfrak{F}_1:&{^a\mathrm{Ind}}_{P_{m,k}(\mathbb{A}_{\mathrm{f}})}^{\mathrm{GL}_{n-1}(\mathbb{A}_{\mathrm{f}})}(\mathcal{W}(\tau_{1,{\mathrm{f}}})\otimes\mathcal{W}(\tau_{2,{\mathrm{f}}}))\\
\longrightarrow &^a\mathrm{Ind}_{P_{m,k}(\mathbb{A}_{\mathrm{f}})}^{\mathrm{GL}_{n-1}(\mathbb{A}_{\mathrm{f}})}\left(H^{b_m}(S_m;V_{1}\otimes M_{\mu_1}^{\vee})(\epsilon')\otimes H^{b_k}(S_k;V_{2}\otimes M_{\mu_2}^{\vee})(\epsilon')\right),
\end{aligned}
\]
which is $\mathrm{Aut}(\C)$-equivariant once we normalize it by $p^{\epsilon'}(\tau_1)^{-1}p^{\epsilon'}(\tau_2)^{-1}$. By \cite[Proposition 4.3]{HarderRaghuram2020book} (and the same argument as in \cite[
(4.34)]{Mahnkopf2005}), we have a map
\[
\begin{aligned}
\mathfrak{F}_2:&^a\mathrm{Ind}_{P_{m,k}(\mathbb{A}_{\mathrm{f}})}^{\mathrm{GL}_{n-1}(\mathbb{A}_{\mathrm{f}})}\left(H^{b_m}(S_m;V_{1}\otimes M_{\mu_1}^{\vee})(\epsilon')\otimes H^{b_k}(S_k;V_{2}\otimes M_{\mu_2}^{\vee})(\epsilon')\right)\\
\longrightarrow&H^{b_{n-1}}(\partial_{P_{m,k}}\widetilde{S}_{n-1},\mathcal{M}_{\lambda}^{\vee})({^a\mathrm{Ind}}_{P_{m,k}(\mathbb{A})}^{\mathrm{GL}_{n-1}(\mathbb{A})}(V_{1}\otimes V_{2})\otimes\epsilon')
\end{aligned}
\]
which is defined over $\Q$. A straightforward verification shows that $\mathcal{F}_{\tau_1,\tau_2,\epsilon'}=\mathfrak{F}_2\circ\mathfrak{F}_1$ which is $\mathrm{Aut}(\C)$-equivariant after normalized by $p^{\epsilon'}(\tau_1)^{-1}p^{\epsilon'}(\tau_2)^{-1}$.
\end{proof}

For the function $f_{\tau_1,\tau_2,{\mathrm{f}}}^{\mathcal{W}}$ chosen in Section \ref{subsection-choiceofvectors}, we define the cohomology classes attached to it by
\begin{equation*}
\theta_{\tau_1,\tau_2,\epsilon'}=\mathcal{F}_{\tau_1,\tau_2,\epsilon'}(f_{\tau_1,\tau_2,{\mathrm{f}}}^{\mathcal{W}}),\,\,\,\theta^0_{\tau_1,\tau_2,\epsilon'}=\mathcal{F}^0_{\tau_1,\tau_2,\epsilon'}(f_{\tau_1,\tau_2,{\mathrm{f}}}^{\mathcal{W}})=p^{\epsilon'}(\tau_1)^{-1}p^{\epsilon'}(\tau_2)^{-1}\theta_{\tau_1,\tau_2,\epsilon'}.
\end{equation*}

\begin{lem}
\label{sectionbehavior}
For any $\sigma\in\mathrm{Aut}(\C)$, we have
\begin{equation*}
^{\sigma}\theta^0_{\tau_1,\tau_2,\epsilon'}=\theta^0_{\tau_1^{\sigma},\tau_2^{\sigma},\epsilon'}.
\end{equation*}
\end{lem}

\begin{proof}
By definition of the cohomology class and the period, it suffices to prove
\[
^{\sigma}f^{\mathcal{W}}_{\tau_1,\tau_2,v}=f^{\mathcal{W}}_{\tau_1^{\sigma},\tau_2^{\sigma},\epsilon'}
\]
for all finite places $v$. This follows immediately by our definition of the function $f^{\mathcal{W}}_{\tau_1,\tau_2,v}$ and the normalization of the new vector.
\end{proof}

\subsection{The Eisenstein cohomology}

Recall that we have an Eisenstein operator
\begin{equation*}
\mathrm{Eis}:\mathrm{Ind}^{\mathrm{GL}_{n-1}(\mathbb{A})}_{P_{m,k}(\mathbb{A})}\left(\tau_1|\cdot|^{s_1-\frac{1}{2}}\otimes\tau_2|\cdot|^{-s_2+\frac{1}{2}}\right)\to\mathcal{A}(\mathrm{GL}_{n-1}(\Q)\backslash\mathrm{GL}_{n-1}(\mathbb{A})),
\end{equation*}
sending a section $f_{s_1,s_2}$ to an Eisenstein series $E(h;f_{s_1,s_2})$. On the cohomology side, we have an Eisenstein operator
\begin{equation*}
\mathrm{Eis}:H^{\bullet}(\partial_{P_{m,k}}\widetilde{S}_{n-1},\mathcal{M}_{\lambda}^{\vee})\to H^{\bullet}(\widetilde{S}_{n-1},\mathcal{M}_{\lambda}^{\vee}),
\end{equation*}
defined by (using the notation of Lie algebra cocycles as in Section \ref{archimedeanpair})
\begin{equation*}
\begin{aligned}
\phi&=\sum_{\boldsymbol{i},b}\omega_{\boldsymbol{i}}\otimes f_{\boldsymbol{i},b}\otimes n_b\\
&\in H^{\bullet}(\mathfrak{g}_{n-1},K_{n-1,\infty};
{^a\mathrm{Ind}}_{P_{m,k}(\mathbb{A})}^{\mathrm{GL}_{n-1}(\mathbb{A})}(\tau_1\otimes\tau_2)\otimes M_{\lambda}^{\vee})\\
\mapsto  \mathrm{Eis}(\phi)&:=\sum_{\gamma\in P_{m,k}(\Q)\backslash\mathrm{GL}_{n-1}(\Q)}\gamma\circ\phi=\sum_{\boldsymbol{i},b}\omega_{\boldsymbol{i}}\otimes E(f_{\boldsymbol{i},b})\otimes n_b.
\end{aligned}
\end{equation*}
Here $f_{\boldsymbol{i},b}\in{^a\mathrm{Ind}}_{P_{m,k}(\mathbb{A})}^{\mathrm{GL}_{n-1}(\mathbb{A})}(\tau_1\otimes\tau_2)$ and 
\begin{equation}
\label{Eis}
E(f_{\boldsymbol{i},b})(g):=\sum_{\gamma\in P_{m,k}(\Q)\backslash\mathrm{GL}_{n-1}(\Q)}f_{\boldsymbol{i},b}(\gamma g)\in\mathcal{A}(\mathrm{GL}_{n-1}(\Q)\backslash\mathrm{GL}_{n-1}(\mathbb{A})).
\end{equation}

We compose the Eisenstein operator with the map \eqref{boundarymap}, to obtain 
\begin{equation*}
\mathrm{Eis}\circ\mathcal{F}_{\tau_1,\tau_2,\epsilon'}:{^a\mathrm{Ind}}_{P_{m,k}(\mathbb{A}_{\mathrm{f}})}^{\mathrm{GL}_{n-1}(\mathbb{A}_{\mathrm{f}})}(\mathcal{W}(\tau_{1,{\mathrm{f}}})\otimes\mathcal{W}(\tau_{2,{\mathrm{f}}}))\to H^{b_{n-1}}(S_{n-1}(R_{\mathrm{f}}),\mathcal{M}_{\lambda}^{\vee}),
\end{equation*}
sending the section
\begin{equation*}
f^{\mathcal{W}}_{\tau_1,\tau_2,{\mathrm{f}}}\mapsto[\theta_{\tau_1,\tau_2,\epsilon'}]\in H^{b_{n-1}}(S_{n-1}(R_{\mathrm{f}}),\mathcal{M}_{\lambda}^{\vee}).
\end{equation*}
We similarly denote
\begin{equation*}
[\theta_{\tau_1,\tau_2,\epsilon'}^0]:=\mathrm{Eis}\circ\mathcal{F}^{\circ}_{\tau_1,\tau_2,\epsilon'}(f^{\mathcal{W}}_{\tau_1,\tau_2,{\mathrm{f}}})\in H^{b_{n-1}}(S_{n-1}(R_{\mathrm{f}}),\mathcal{M}_{\lambda}^{\vee}),
\end{equation*}
for the normalized cohomology class when composite with the map \eqref{normalizeboundary}. 

The canonical map $p: F_{n-1}(R_{\mathrm{f}})\to S_{n-1}(R_{\mathrm{f}})$ induces a mapping
\begin{equation*}
p^{\ast}:H^{b_{n-1}}(S_{n-1}(R_{\mathrm{f}}),\mathcal{M}_{\lambda}^{\vee})\to H^{b_{n-1}}(F_{n-1}(R_{\mathrm{f}}),\mathcal{M}_{\lambda}^{\vee}),
\end{equation*}
which provide us cohomology classes
\begin{equation*}
p^{\ast}[\theta_{\tau_1,\tau_2,\epsilon'}]\in H^{b_{n-1}}(F_{n-1}(R_{\mathrm{f}}),\mathcal{M}_{\lambda}^{\vee}),\qquad p^{\ast}[\theta_{\tau_1,\tau_2,\epsilon'}^0]\in H^{b_{n-1}}(F_{n-1}(R_{\mathrm{f}}),\mathcal{M}_{\lambda}^{\vee}).
\end{equation*}

\begin{lem}
\label{taubehavior}
For any $\sigma\in\mathrm{Aut}(\C)$, we have
\begin{equation*}
^{\sigma}[\theta_{\tau_1,\tau_2,\epsilon'}^0]=[\theta_{\tau^{\sigma}_1,\tau^{\sigma}_2,\epsilon'}^0].
\end{equation*}
\end{lem}

\begin{proof}
We generalize the argument in \cite[Section 4.5, 4.6, 4.7]{Mahnkopf2005} to Eisenstein series induced from any parabolics using the supplement of \cite{HarderRaghuram2020book}. 

As in \cite[Section 4.5]{Mahnkopf2005}, we characterize a subspace
\[
H^{\bullet}_{\succeq}(\partial_{P_{m,k}}S_{n-1},\mathcal{M}^{\vee}_{\lambda})\subset H^{\bullet}_{\mathrm{cusp}}(\partial_{P_{m,k}}S_{n-1},\mathcal{M}^{\vee}_{\lambda})
\]
such that $\mathrm{Eis}$ is well defined on $H^{\bullet}_{\succeq}(\partial_{P_{m,k}}S_{n-1},\mathcal{M}^{\vee}_{\lambda})$. More precisely, by \cite[(5.18)]{HarderRaghuram2020book} (which is a generalization of \cite[Theorem 4.2]{Mahnkopf2005} for general parabolic subgroups), we can write
\[
\begin{aligned}
H^{\bullet}_{\mathrm{cusp}}(\partial_{P_{m,k}}S_{n-1},\mathcal{M}^{\vee}_{\lambda})=H^{\bullet}_{\succeq}(\partial_{P_{m,k}}S_{n-1},\mathcal{M}^{\vee}_{\lambda})\oplus H^{\bullet}_{\preceq}(\partial_{P_{m,k}}S_{n-1},\mathcal{M}^{\vee}_{\lambda}).
\end{aligned}
\]
That is we are taking $H^{\bullet}_{\succeq}$ (resp. $H^{\bullet}_{\preceq}$) to be the contribution of the direct summand in the second (resp. third) line of \cite[(5.18)]{HarderRaghuram2020book}. The same proof in \cite[Proposition 4.6]{Mahnkopf2005} implies that 
\[
\mathrm{Eis}(\phi)^{\sigma}=\mathrm{Eis}(\phi^{\sigma}),\qquad \sigma\in\mathrm{Aut}(\C).
\]
for $\phi\in H^{\bullet}_{\succeq}(\partial_{P_{m,k}}S_{n-1},\mathcal{M}^{\vee}_{\lambda})$ where, in the case $m=k$, we assume $\mathrm{Eis}(\phi)$ is not a residual Eisenstein series.

We remark that our assumption that the weights $\mu_1,\mu_2$ are balanced implies that $\theta_{\tau_1,\tau_2}\in H^{\bullet}_{\succeq}(\partial_{P_{m,k}}S_{n-1},\mathcal{M}^{\vee}_{\lambda})$. Furthermore, by \cite[Proposition 7.10]{HarderRaghuram2020book} (exchanging $\tau_1,\tau_2$ and $m,k$ if necessary), $E(f^{\mathcal{W}}_{\tau_1,\tau_2})$ is not a residual Eisenstein series when $m=k$. Therefore, by Lemma \ref{sectionbehavior},
\[
\begin{aligned}
^{\sigma}[\theta_{\tau_1,\tau_2,\epsilon'}^0]=\mathrm{Eis}({^{\sigma}\theta}^0_{\tau_1,\tau_2,\epsilon'})=\mathrm{Eis}(\theta^0_{\tau_1^{\sigma},\tau_2^{\sigma},\epsilon'})=[\theta^0_{\tau_1^{\sigma},\tau_2^{\sigma},\epsilon'}],
\end{aligned}
\]
as desired. 
\end{proof}

\section{The Cohomological Interpretation of the Integral}
\label{cohomologicalinterpretation}

We are now going to provide a cohomological interpretation of the global integral \eqref{eq-globalintegral}. We keep all the notations and assumptions as in previous sections and further fix the permissible signs $\epsilon,\epsilon'$ such that $\epsilon=(-1)^n\epsilon'$. Note that this uniquely determines $\epsilon,\epsilon'$ since at least one of $n,m,k$ is odd.

\subsection{The pairing}
\label{pairing}

In Section \ref{cohomologyclass}, we obtain two cohomology classes 
\begin{equation*}
\iota^{\ast}\theta_{\pi,\epsilon}\in H_c^{b_n}(F_{n-1}(R_{\mathrm{f}}),\iota^{\ast}\mathcal{M}_{\mu}^{\vee}),\qquad p^{\ast}[\theta_{\tau_1,\tau_2,\epsilon'}]\in H^{b_{n-1}}(F_{n-1}(R_{\mathrm{f}}),\mathcal{M}_{\lambda}^{\vee}).
\end{equation*}
We review the pairing of these two cohomology classes defined in \cite[Section 5.1.2]{Mahnkopf2005} and \cite[Section 3.2.4]{Raghuram2010}.

We assume $\mu^{\vee}\succ\lambda$. That is, $a_1\geq b_1\geq a_2\geq b_2\geq...\geq b_{n-1}\geq a_n$ if we write $\mu^{\vee}=(a_1,...,a_n)$ and $\lambda=(b_1,...,b_{n-1})$. This assumption implies that $M_{\lambda}$ appears (with multiplicity one) in $\iota^{\ast}M_{\mu}^{\vee}$.

The $\mathrm{GL}_{n-1}$-equivariant pairing
\begin{equation*}
\langle\cdot,\cdot\rangle:\iota^{\ast}M_{\mu}^{\vee}\times M_{\lambda}^{\vee}\to\C
\end{equation*}
induces the pairing of sheaves
\begin{equation*}
\langle\cdot,\cdot\rangle:\iota^{\ast}\mathcal{M}_{\mu}^{\vee}\otimes\mathcal{M}_{\lambda}^{\vee}\to\underline{\C}
\end{equation*}
and thus (composite with the cup/wedge product) provides the pairing 
\begin{equation*}
\begin{aligned}
\langle\cdot,\cdot\rangle_{\mathcal{C}}:H_c^{b_n}(F_{n-1}(R_{\mathrm{f}}),\iota^{\ast}\mathcal{M}_{\mu}^{\vee})\times H^{b_{n-1}}(F_{n-1}(R_{\mathrm{f}}),\mathcal{M}_{\lambda}^{\vee})&\to H_c^{d_{n-1}}(F_{n-1}(R_{\mathrm{f}}),\underline{\C})\\
&\to \C
\end{aligned}
\end{equation*}
where $\mathcal{C}$ is the Harder-Mahnkopf cycle (associated to $R_f$) as in \cite[Section 5.1.1]{Mahnkopf2005} and \cite[Section 3.2.3]{Raghuram2010}.

We write
\begin{equation*}
\langle\theta_{\pi,\epsilon},[\theta_{\tau_1,\tau_2,\epsilon'}]\rangle:=\langle\iota^{\ast}\theta_{\pi,\epsilon},p^{\ast}[\theta_{\tau_1,\tau_2,\epsilon'}]\rangle_{\mathcal{C}}=\int_{\mathcal{C}}\iota^{\ast}\theta_{\pi,\epsilon}\wedge p^{\ast}[\theta_{\tau_1,\tau_2,\epsilon'}],
\end{equation*}
and observe that it is stable under the action of $\pi_0(\mathrm{GL}_{n-1})$ due to our choice of $\epsilon,\epsilon'$.

\subsection{The archimedean pairing and the nonvanishing hypothesis}
\label{archimedeanpair}

Denote $\mathfrak{k}_{n}$ be the Lie algebra of $K_{n,\infty}$. We fix a basis $\{\omega_i'\}$ of $(\mathfrak{g}_n/\mathfrak{k}_{n})^{\ast}$ and a basis $\{\omega_i\}$ of $(\mathfrak{g}_{n-1}/\mathfrak{k}_{n-1})^{\ast}$ such that $\iota^{\ast}\omega_i'=\omega_i$ for $1\leq i\leq \frac{n(n-1)}{2}$ and $\iota^{\ast}\omega_i'=0$ if $i>\frac{n(n-1)}{2}$, where $\iota:\mathfrak{g}_{n-1}/\mathfrak{k}_{n-1}\to\mathfrak{g}_{n}/\mathfrak{k}_{n}$ is the inclusion. Let $\{m_a\}$ (resp. $\{n_b\}$) be a $\Q$-basis for $M_{\mu}^{\vee}$ (resp. $M_{\lambda}^{\vee}$). 

Our fixed generator $[\pi_{\infty}]$ is represented by a $K_{n,\infty}^0$-invariant element in $\wedge^{b_n}(\mathfrak{g}_{n}/\mathfrak{k}_{n})^{\ast}\otimes\mathcal{W}(\pi_{\infty})\otimes M_{\mu}^{\vee}$:
\begin{equation*}
[\pi_{\infty}]=\sum_{\boldsymbol{j}=j_1<...<j_{b_n}}\sum_a\omega'_{\boldsymbol{j}}\otimes W_{\infty,\boldsymbol{j},a}\otimes m_a,
\end{equation*}
where $W_{\infty,\boldsymbol{j},a}\in\mathcal{W}(\pi_{\infty})$ is an archimedean Whittaker function.

Similarly, our fixed generator $[\tau_{1,\infty}]\otimes[\tau_{2,\infty}]$ is a represented by a $K_{n-1,\infty}^0$-invariant element in $\wedge^{b_{n-1}}(\mathfrak{g}_{n-1}/\mathfrak{k}_{n-1})^{\ast}\otimes{^a\mathrm{Ind}}_{P_{m,k}(\R)}^{\mathrm{GL}_{n-1}(\R)}(\mathcal{W}(\tau_{1,\infty})\otimes\mathcal{W}(\tau_{2,\infty}))\otimes M_{\lambda}^{\vee}$:
\begin{equation*}
[\tau_{1,\infty}]\otimes[\tau_{2,\infty}]=\sum_{\boldsymbol{i}=i_1<...<i_{b_{n-1}}}\sum_b\omega_{\boldsymbol{i}}\otimes f_{\infty,\boldsymbol{i},b}\otimes n_b,
\end{equation*}
where $f_{\infty,\boldsymbol{i},b}\in{^a\mathrm{Ind}}_{P_{m,k}(\R)}^{\mathrm{GL}_{n-1}(\R)}(\mathcal{W}(\tau_{1,\infty})\otimes\mathcal{W}(\tau_{2,\infty}))$.

We define a pairing at infinity by
\begin{equation*}
\left\langle[\pi_{\infty}],[\tau_{1,\infty}]\otimes[\tau_{2,\infty}]\right\rangle=\sum_{\boldsymbol{i,j}}\mathrm{sgn}(\boldsymbol{i},\boldsymbol{j})\sum_{a,b}\langle n_b,m_a\rangle \Psi_{\infty}(W_{\infty,\boldsymbol{j},a},f_{\infty,\boldsymbol{i},b}).
\end{equation*}

We need the following nonvanishing hypothesis 
\begin{itemize}
\item{$\left\langle[\pi_{\infty}],[\tau_{1,\infty}]\otimes[\tau_{2,\infty}]\right\rangle\neq 0$},
\end{itemize}
and then we can set an archimedean period
\begin{equation}
\label{archimedeanperiod}
p_{\infty}(\pi,\tau_1,\tau_2):=p_{\infty}(\mu,\mu_1,\mu_2):=\frac{1}{\left\langle[\pi_{\infty}],[\tau_{1,\infty}]\otimes[\tau_{2,\infty}]\right\rangle}.
\end{equation}

This assumption, also appeared in \cite[Hypothesis 3.10]{Raghuram2010} and \cite{Mahnkopf1998, Mahnkopf2005}, is a standard technical problem in the study of special $L$-values (see also \cite{RaghuramShahidi, RaghuramShahidi2008, KastenSchmidt, KazhdanMazurSchmidt}). In the case for special $L$-values of $\mathrm{GL}_n\times\mathrm{GL}_{n-1}$, the nonvanishing hypothesis is proved by \cite{Sun}. See \cite[Theorem 3.8]{GrobnerHarris} for an explanation on how one can deduce the nonvanishing hypothesis from Sun's result \cite[Theorem A]{Sun}. In the case treated in this paper, the above nonvanishing hypothesis is proved in \cite{LLS22}.

\subsection{The main identity}

We now give a cohomological interpretation of the integral \eqref{eq-globalintegral} and \eqref{integralrepresentation}.

\begin{prop}
\label{mainidentity}
We keep all the assumptions as previous, then
\begin{equation*}
    \langle\theta^0_{\pi,\epsilon},[\theta^0_{\tau_1,\tau_2,\epsilon'}]\rangle=\frac{\mathrm{vol}(K_0(\mathfrak{f}))}{p^{\epsilon}(\pi)p^{\epsilon'}(\tau_1)p^{\epsilon'}(\tau_2)p_{\infty}(\mu,\mu_1,\mu_2)}\cdot\frac{L^{S}(\frac{1-k}{2},\pi\times\tau_1)L^{S}(\frac{1-m}{2},\widetilde{\pi}\times\widetilde{\tau_2})}{L^{S}(1+\frac{1-n}{2},\tau_1\times\widetilde{\tau_2})}.
\end{equation*}
\end{prop}

\begin{proof}
The class $\theta_{\pi,\epsilon}\in H^{b_n}(\mathfrak{g}_n,K_{n,\infty}^0;V_{\pi}\otimes M_{\mu}^{\vee})(\epsilon)$ is represented by a $K_{n,\infty}^0$-invariant element in $\wedge^{b_n}(\mathfrak{g}_n/\mathfrak{k}_{n})\otimes V_{\pi}\otimes M_{\mu}^{\vee}$:
\[
\theta_{\pi,\epsilon}=\sum_{\boldsymbol{j}=j_1<...<j_{b_n}}\sum_a\omega_{\boldsymbol{j}}'\otimes\phi_{\boldsymbol{j},a}\otimes m_a
\]
with $\phi_{\boldsymbol{j},a}\in V_{\pi}$. 

The class $\theta_{\tau_1,\tau_2,\epsilon'}\in H^{b_{n-1}}(\mathfrak{g}_{n-1},K_{n-1,\infty}^0;{^a\mathrm{Ind}}_{P_{m,k}(\mathbb{A})}^{\mathrm{GL}_{n-1}(\mathbb{A})}(V_1\otimes V_2)\otimes M_{\lambda}^{\vee})(\epsilon')$ is represented by a $K_{n-1,\infty}^0$-invariant in $\wedge^{b_{n-1}}(\mathfrak{g}_{n-1}/\mathfrak{k}_{n-1})\otimes {^a\mathrm{Ind}}_{P_{m,k}(\mathbb{A})}^{\mathrm{GL}_{n-1}(\mathbb{A})}(V_1\otimes V_2)\otimes M_{\lambda}^{\vee}$:
\[
\theta_{\tau_1,\tau_2,\epsilon'}=\sum_{\boldsymbol{i}=i_1<...<i_{b_{n-1}}}\sum_b\omega_{\boldsymbol{i}}\otimes f_{\tau_1,\tau_2,\boldsymbol{i},b}\otimes n_b,
\]
with $f_{\tau_1,\tau_2,\boldsymbol{i},b}\in{^a\mathrm{Ind}}_{P_{m,k}(\mathbb{A})}^{\mathrm{GL}_{n-1}(\mathbb{A})}(V_1\otimes V_2)$. Then the class $[\theta_{\tau_1,\tau_2,\epsilon'}]$ is given by
\[
[\theta_{\tau_1,\tau_2,\epsilon'}]=\sum_{\boldsymbol{i}=i_1<...<i_{b_{n-1}}}\sum_b\omega_{\boldsymbol{i}}\otimes E(f_{\tau_1,\tau_2,\boldsymbol{i},b})\otimes n_b,
\]
where $E(f_{\tau_1,\tau_2,\boldsymbol{i},b})$ is defined as in \eqref{Eis}.

We calculate that
\[
\begin{aligned}
&\langle\theta_{\pi,\epsilon},[\theta_{\tau_1,\tau_2,\epsilon'}]\rangle\\
=&\frac{1}{\mathrm{vol}(R_{\mathrm{f}})}\sum_{\boldsymbol{i,j},a,b}\mathrm{sgn}(\boldsymbol{i},\boldsymbol{j})\langle n_b,n_a\rangle\int_{\mathrm{GL}_{n-1}(\Q)\backslash\mathrm{GL}_{n-1}(\mathbb{A})/K^1_{n-1,\infty}R_{\mathrm{f}}}\phi_{\boldsymbol{j},a}(\iota(g))E(f_{\tau_1,\tau_2,\boldsymbol{i},b})(g)dg\\
=&\sum_{\boldsymbol{i,j},a,b}\mathrm{sgn}(\boldsymbol{i},\boldsymbol{j})\langle n_b,n_a\rangle\int_{\mathrm{GL}_{n-1}(\Q)\backslash\mathrm{GL}_{n-1}(\mathbb{A})}\phi_{\boldsymbol{j},a}(\iota(g))E(f_{\tau_1,\tau_2,\boldsymbol{i},b})(g)dg\\
=&\sum_{\boldsymbol{i,j},a,b}\mathrm{sgn}(\boldsymbol{i},\boldsymbol{j})\langle n_b,n_a\rangle\cdot I(\phi_{\boldsymbol{j},a},f_{\tau_1,\tau_2,\boldsymbol{i},b}).
\end{aligned}
\]
The integral representation then provide us
\[
\begin{aligned}
    \langle\theta_{\pi,\epsilon},[\theta_{\tau_1,\tau_2,\epsilon'}]\rangle=\mathrm{vol}(K_0(\mathfrak{f}))\cdot\frac{L^{S}(\frac{1-k}{2},\pi\times\tau_1)L^{S}(\frac{1-m}{2},\widetilde{\pi}\times\widetilde{\tau_2})}{L^{S}(1+\frac{1-n}{2},\tau_1\times\widetilde{\tau_2})}\cdot\langle[\pi_{\infty}],[\tau_{1,\infty}]\otimes[\tau_{2,\infty}]\rangle,
\end{aligned}
\]
which completes the proof of the proposition.
\end{proof}

\section{Special Values of $L$-functions}
\label{section-mainthm}

We now put all the assumptions together and complete the proof of the main theorem.

Let $\pi$ (resp. $\tau_1$, resp. $\tau_2$) be a regular algebraic cuspidal automorphic representation of $\mathrm{GL}_n(\mathbb{A})$ (resp. $\mathrm{GL}_m(\mathbb{A})$, resp. $\mathrm{GL}_k(\mathbb{A})$) with $n=m+k+1$ and $mk$ is even. Then there is a dominant integral pure weight $\mu$ (resp. $\mu_1$, resp. $\mu_2$) such that $\pi\in\mathrm{Coh}(\mathrm{GL}_n,\mu^{\vee})$ (resp. $\tau_1\in\mathrm{Coh}(\mathrm{GL}_m,\mu_1^{\vee})$, resp. $\tau_2\in\mathrm{Coh}(\mathrm{GL}_k,\mu_2^{\vee})$). We assume:
\begin{itemize}
\item{$s=\frac{1-k}{2}$ is a critical value for $L(s,\pi\times\widetilde{\tau_1})$,}
\item{$s=\frac{1-m}{2}$ is a critical value for $L(s,\widetilde{\pi}\times\widetilde{\tau_2})$,}
\item{$s=\frac{1-n}{2}$ and $s=1+\frac{1-n}{2}$ are critical values for $L(s,\tau_1\times\widetilde{\tau_2})$.}
\end{itemize}
The reader can refer to \cite[Proposition 7.7]{HarderRaghuram2020book} for the critical set of the Rankin-Selberg $L$-function and deduce more explicit conditions on $\mu,\mu_1,\mu_2$ for our assumptions. Note that by \cite[Lemma 7.14]{HarderRaghuram2020book}, the third assumption above guarantees that $\mu_1,\mu_2$ are balanced as we assumed in Section \ref{cohomologyinduced}. That is there exists a balanced Kostant representative $w$ such that $\lambda=w^{-1}(\mu_1+\mu_2)$. As in Section \ref{pairing}, we further assume that
\begin{itemize}
\item{$\mu^{\vee}\succ\lambda$.}
\end{itemize}
Moreover by \cite[Proposition 7.10]{HarderRaghuram2020book}, without loss of generality (by exchanging $\tau_1,\tau_2$ and $m,k$ if necessary), we can assume that our Eisenstein series $E(h;f_{s_1,s_2})$ is holomorphic at $s_1=\frac{1-k}{2}$ and $s_2=\frac{1-m}{2}$.

To ease our notation and presentation, we have assumed $\tau_1,\tau_2$ have trivial central characters. This assumption is only for simplicity and they will appear as Gauss sums (see for example \cite[Theorem 1.1]{Raghuram2010}) in our main theorem if we drop this assumption. We fix permissible signs $\epsilon,\epsilon'$ such that $\epsilon=(-1)^n\epsilon'$ and let $p^{\epsilon}(\pi), p^{\epsilon'}(\tau_1), p^{\epsilon'}(\tau_2)$ be periods defined in the Section \ref{cohomologyclass}. Set $p_{\infty}(\mu,\mu_1,\mu_2)$ be the archimedean period as in \eqref{archimedeanperiod}. Denote $S$ for a finite set of places including the archimedean place $\infty$ and all places such that $\pi_v$ or $\tau_{1,v}$ or $\tau_{2,v}$ is ramified. We state our main theorem in the following.

\begin{thm}
\label{maintheorem}
Under all above assumptions, for any $\sigma\in\mathrm{Aut}(\C)$, we have
\begin{equation*}
\begin{aligned}
&\sigma\left(\frac{L^S(\frac{1-k}{2},\pi\times\tau_1)L^S(\frac{1-m}{2},\widetilde{\pi}\times\widetilde{\tau_2})}{L^S(1+\frac{1-n}{2},\tau_1\times\widetilde{\tau_2})p^{\epsilon}(\pi)p^{\epsilon'}(\tau_1)p^{\epsilon'}(\tau_2)p_{\infty}(\mu,\mu_1,\mu_2)}\right)\\
=&\frac{L^S(\frac{1-k}{2},\pi^{\sigma}\times\tau^{\sigma}_1)L^S(\frac{1-m}{2},\widetilde{\pi}^{\sigma}\times\widetilde{\tau_2}^{\sigma})}{L^S(1+\frac{1-n}{2},\tau^{\sigma}_1\times\widetilde{\tau_2}^{\sigma})p^{\epsilon}(\pi^{\sigma})p^{\epsilon'}(\tau^{\sigma}_1)p^{\epsilon'}(\tau^{\sigma}_2)p_{\infty}(\mu,\mu_1,\mu_2)}.
\end{aligned}    
\end{equation*}
In particular,
\begin{equation*}
\frac{L^S(\frac{1-k}{2},\pi\times\tau_1)L^S(\frac{1-m}{2},\widetilde{\pi}\times\widetilde{\tau_2})}{L^S(1+\frac{1-n}{2},\tau_1\times\widetilde{\tau_2})p^{\epsilon}(\pi)p^{\epsilon'}(\tau_1)p^{\epsilon'}(\tau_2)p_{\infty}(\mu,\mu_1,\mu_2)}\in\Q(\pi,\tau_1,\tau_2).
\end{equation*}
\end{thm}

\begin{proof}
Note that the pairing $\langle\cdot,\cdot\rangle_{\mathcal{C}}$ and $\iota^{\ast},p^{\ast}$ are $\mathrm{Aut}(\C)$-equivariant (see for example \cite[Lemma 1.2]{Mahnkopf2005}). For any $\sigma\in\mathrm{Aut}(\C)$, we have
\[
\sigma(\langle\theta_{\pi,\epsilon}^0,[\theta^0_{\tau_1,\tau_2,\epsilon'}]\rangle)=\langle{^{\sigma}\theta^0_{\pi,\epsilon}},{^{\sigma}[\theta^0_{\tau_1,\tau_2,\epsilon'}]}\rangle=\langle\theta^0_{\pi^{\sigma},\epsilon},[\theta^0_{\tau_1^{\sigma},\tau_2^{\sigma},\epsilon'}]\rangle.
\]
Here the second equality follows from Lemma \ref{pibehavior} and Lemma \ref{taubehavior}. The theorem then follows from the main identity \eqref{mainidentity} and the fact that $\mathrm{vol}(K_0(\mathfrak{f}))\in\Q$.
\end{proof}

The above theorem is stated for the partial $L$-functions. We can also obtain the result of the whole finite $L$-function by applying the technique in \cite{Raghuram2010}.

\begin{cor}
\label{finiteL}
We keep all the assumptions as above. For any $\sigma\in\mathrm{Aut}(\C)$, let $\rho_{\sigma}=\prod_{v}\rho_{\sigma,v}$ be a quadratic character such that $\rho_{\sigma,v}(x)=\frac{\sigma(|x|^{1/2})}{|x|^{1/2}}$ for all finite places $v\in S$ and $\rho_{\sigma,v}=1$ for all $v\notin S$.\\
(1) If both $m,k$ are even, and $n$ is odd, then for any $\sigma\in\mathrm{Aut}(\C)$,
\begin{equation*}
\begin{aligned}
&\sigma\left(\frac{L_{\mathrm{f}}(\frac{1-k}{2},\pi\times\tau_1)L_{\mathrm{f}}(\frac{1-m}{2},\widetilde{\pi}\times\widetilde{\tau_2})}{L_{\mathrm{f}}(1+\frac{1-n}{2},\tau_1\times\widetilde{\tau_2})p^{\epsilon}(\pi)p^{\epsilon'}(\tau_1)p^{\epsilon'}(\tau_2)p_{\infty}(\mu,\mu_1,\mu_2)}\right)\\
=&\frac{L_{\mathrm{f}}(\frac{1-k}{2},\pi^{\sigma}\times\tau^{\sigma}_1)L_{\mathrm{f}}(\frac{1-m}{2},\widetilde{\pi}^{\sigma}\times\widetilde{\tau_2}^{\sigma})}{L_{\mathrm{f}}(1+\frac{1-n}{2},\tau^{\sigma}_1\times\widetilde{\tau_2}^{\sigma}\times\rho_{\sigma})p^{\epsilon}(\pi^{\sigma})p^{\epsilon'}(\tau^{\sigma}_1)p^{\epsilon'}(\tau^{\sigma}_2)p_{\infty}(\mu,\mu_1,\mu_2)}.
\end{aligned}    
\end{equation*}
(2) If $k$ is even, $m$ is odd and $n$ is even (the case $k$ odd and $m$ even is similar), then for any $\sigma\in\mathrm{Aut}(\C)$,
\begin{equation*}
\begin{aligned}
&\sigma\left(\frac{L_{\mathrm{f}}(\frac{1-k}{2},\pi\times\tau_1)L_{\mathrm{f}}(\frac{1-m}{2},\widetilde{\pi}\times\widetilde{\tau_2})}{L_{\mathrm{f}}(1+\frac{1-n}{2},\tau_1\times\widetilde{\tau_2})p^{\epsilon}(\pi)p^{\epsilon'}(\tau_1)p^{\epsilon'}(\tau_2)p_{\infty}(\mu,\mu_1,\mu_2)}\right)\\
=&\frac{L_{\mathrm{f}}(\frac{1-k}{2},\pi^{\sigma}\times\tau^{\sigma}_1)L_{\mathrm{f}}(\frac{1-m}{2},\widetilde{\pi}^{\sigma}\times\widetilde{\tau_2}^{\sigma}\times\rho_{\sigma})}{L_{\mathrm{f}}(1+\frac{1-n}{2},\tau^{\sigma}_1\times\widetilde{\tau_2}^{\sigma})p^{\epsilon}(\pi^{\sigma})p^{\epsilon'}(\tau^{\sigma}_1)p^{\epsilon'}(\tau^{\sigma}_2)p_{\infty}(\mu,\mu_1,\mu_2)}.
\end{aligned}    
\end{equation*}
\end{cor}

\begin{proof}
The corollary follows by applying \cite[Proposition 3.17]{Raghuram2010}. We note that the quadratic character $\rho_{\sigma}$ appears in the identity \cite[(3.20)]{Raghuram2010}
\[
(\pi_{1,v}\boxtimes\pi_{2,v})^{\sigma}=(\pi_{1,v}^{\sigma}\boxtimes\pi_{2,v}^{\sigma})\otimes\rho_{\sigma,v}^{(1-m_1)(1-m_2)}
\]
where $\pi_{1,v}$ (resp. $\pi_{2,v}$) is an irreducible admissible representation of $\mathrm{GL}_{m_1}(\Q_v)$ (resp. $\mathrm{GL}_{m_2}(\Q_v)$). Especially, the character $\rho_{\sigma,v}$ will be cancelled if one of $m_1,m_2$ is odd.
\end{proof}

Using the period relations in \cite[Theorem 1.1]{RaghuramShahidi2008}, we may also obtain results for twisted $L$-functions and other critical values. For example, 

\begin{cor}
\label{othercriticalvalues}
Assume our above assumptions on weights hold for $\mu_1,\mu_2$ replaced by $\mu_1'=\mu_1+k,\mu_2'=\mu_2-m$. We change our assumptions on critical points according to following two cases.\\
(1) For $m,k$ even and $n$ odd, we assume:
\begin{itemize}
\item{$s=\frac{1}{2}$ is a critical value for $L(s,\pi\times\widetilde{\tau_1})$,}
\item{$s=\frac{1}{2}$ is a critical value for $L(s,\widetilde{\pi}\times\widetilde{\tau_2})$,}
\item{$s=0$ and $s=1$ are critical values for $L(s,\tau_1\times\widetilde{\tau_2})$.}
\end{itemize}
Then for any $\sigma\in\mathrm{Aut}(\C)$,
\begin{equation*}
    \begin{aligned}
&\sigma\left(\frac{L^S(\frac{1}{2},\pi\times\tau_1)L^S(\frac{1}{2},\widetilde{\pi}\times\widetilde{\tau_2})}{L^S(1,\tau_1\times\widetilde{\tau_2})p^{\epsilon}(\pi)p^{\epsilon'}(\tau_1)p^{\epsilon'}(\tau_2)p_{\infty}(\mu,\mu'_1,\mu'_2)}\right)\\
=&\frac{L^S(\frac{1}{2},\pi^{\sigma}\times\tau^{\sigma}_1)L^S(\frac{1}{2},\widetilde{\pi}^{\sigma}\times\widetilde{\tau_2}^{\sigma})}{L^S(1,\tau^{\sigma}_1\times\widetilde{\tau_2}^{\sigma})p^{\epsilon}(\pi^{\sigma})p^{\epsilon'}(\tau^{\sigma}_1)p^{\epsilon'}(\tau^{\sigma}_2)p_{\infty}(\mu,\mu'_1,\mu'_2)}.
\end{aligned}    
\end{equation*}
(2) For $k$ even, $m$ odd and $n$ even (the case $k$ odd and $m$ even is similar), we assume:
\begin{itemize}
\item{$s=\frac{1}{2}$ is a critical value for $L(s,\pi\times\widetilde{\tau_1})$,}
\item{$s=0$ is a critical value for $L(s,\widetilde{\pi}\times\widetilde{\tau_2})$,}
\item{$s=-\frac{1}{2}$ and $s=\frac{1}{2}$ are critical values for $L(s,\tau_1\times\widetilde{\tau_2})$.}
\end{itemize}
Then for any $\sigma\in\mathrm{Aut}(\C)$,
\begin{equation*}
    \begin{aligned}
&\sigma\left(\frac{L^S(\frac{1}{2},\pi\times\tau_1)L^S(0,\widetilde{\pi}\times\widetilde{\tau_2})}{L^S(\frac{1}{2},\tau_1\times\widetilde{\tau_2})p^{\epsilon}(\pi)p^{\epsilon'}(\tau_1)p^{-\epsilon'}(\tau_2)p_{\infty}(\mu,\mu'_1,\mu'_2)}\right)\\
=&\frac{L^S(\frac{1}{2},\pi^{\sigma}\times\tau^{\sigma}_1)L^S(0,\widetilde{\pi}^{\sigma}\times\widetilde{\tau_2}^{\sigma})}{L^S(\frac{1}{2},\tau^{\sigma}_1\times\widetilde{\tau_2}^{\sigma})p^{\epsilon}(\pi^{\sigma})p^{\epsilon'}(\tau^{\sigma}_1)p^{-\epsilon'}(\tau^{\sigma}_2)p_{\infty}(\mu,\mu'_1,\mu'_2)}.
\end{aligned}    
\end{equation*}
\end{cor}

\begin{proof}
We apply Theorem \ref{maintheorem} for $\tau_1,\tau_2$ replaced by $\tau_1|\cdot|^k,\tau_2|\cdot|^{-m}$. This change the weight $\mu_1,\mu_2$ to $\mu_1',\mu_2'$. The period relation \cite[Theorem 1.1]{RaghuramShahidi2008} implies
\[
p^{\epsilon'}(\tau_1|\cdot|^k)=p^{(-1)^k\epsilon'}(\tau_1),\qquad p^{\epsilon'}(\tau_2|\cdot|^k)=p^{(-1)^m\epsilon'}(\tau_2)
\]
which completes the proof of the corollary.
\end{proof}

We can also obtain the similar result for the whole finite $L$-function as in Corollary \ref{finiteL}. We finally remark that in the case $k$ even, $m$ odd and $n$ even, if we set
\begin{equation}
\label{inductionperiod}
\Omega(\pi,\tau_1,\tau_2)=\frac{L^S(0,\widetilde{\pi}\times\widetilde{\tau_2})}{p^{\epsilon}(\pi)p^{\epsilon'}(\tau_1)p^{-\epsilon'}(\tau_2)p_{\infty}(\mu,\mu'_1,\mu'_2)},
\end{equation}
then we obtain an algebraicity result for
\begin{equation*}
    \frac{1}{\Omega(\pi,\tau_1,\tau_2)}\cdot\frac{L^S(\frac{1}{2},\pi\times\tau_1)}{L^S(\frac{1}{2},\tau_1\times\widetilde{\tau_2})}.
\end{equation*}
Note that here $L^S(0,\widetilde{\pi}\times\widetilde{\tau_2})$ is nonzero by \cite{JS}. If we apply the induction process as in \cite{Mahnkopf2005} (starting from $\mathrm{GL}_{m-1}\times\mathrm{GL}_m$ and $\mathrm{GL}_{m+1}\times\mathrm{GL}_m$), then we can obtain an algebraicity result for $\mathrm{GL}_n\times\mathrm{GL}_m$ where $m$ is any odd integer and $n$ of the form $m-1+l(m+1)$ or $m+1+l(m+1)$ with $l$ any non-negative integer. In particular, when $m=1$ or $m=3$ we can take $n$ to be any even integer.

\bibliographystyle{alpha}
\bibliography{References}

\end{document}